\documentclass[]{amsart}
\usepackage{amsthm}
\usepackage{tikz}
\usetikzlibrary{cd}
\usepackage{amssymb}
\usepackage{amsmath}
\usepackage[margin=1.2in]{geometry}
\usepackage{enumitem}
\usepackage{hyperref}

\usepackage{todonotes}
\theoremstyle{plain}
\newtheorem{theorem}{Theorem}[section]
\newtheorem{proposition}[theorem]{Proposition}
\newtheorem{lemma}[theorem]{Lemma}
\newtheorem{corollary}[theorem]{Corollary}
\newtheorem{bigthm}{Theorem}

\theoremstyle{definition}
\newtheorem{definition}[theorem]{Definition}
\newtheorem{example}[theorem]{Example}
\newtheorem{remark}[theorem]{Remark}
\newcommand{\inv}{^{-1}}

\newcommand{\Z}{\mathbb{Z}}
\newcommand{\R}{\mathbb{R}}

\newcommand{\F}[1]{\mathbb{F}_{#1}}
\renewcommand{\hom}{\mathop{\mathrm{Hom}}\nolimits}
\DeclareMathOperator{\im}{im}
\DeclareMathOperator{\Har}{Har}
\DeclareMathOperator{\rank}{rank}
\DeclareMathOperator{\nullity}{null}
\DeclareMathOperator{\coker}{coker}

\newcommand{\xto}[1]{\overset{#1}{\to}}
\newcommand{\bF}{\mathbb{F}}

\begin{document}

\author{Michael J. Catanzaro and Brantley Vose}

\title{Harmonic Representatives in Homology over Arbitrary Fields}
\begin{abstract}
  We introduce a notion of harmonic chain for chain complexes over fields of positive 
  characteristic. A list of conditions for when a Hodge decomposition theorem holds in
  this setting is given and we apply this theory to finite CW complexes. An explicit construction
  of the harmonic chain within a homology class is described when applicable. We show how the coefficients
  of usual discrete harmonic chains due to Eckmann can be reduced to localizations of 
  the integers, allowing us to compare classical harmonicity with the notion introduced here. 
  We focus on applications throughout, including CW decompositions of orientable surfaces and examples
  of spaces arising from sampled data sets.
\end{abstract}

\maketitle

\section{Introduction}\label{sec:intro}
The classical Hodge decomposition theorem~\cite{hodge_theory_1989,
warner_foundations_1983} states that the set of all differential $k$-forms on a
smooth, compact, Riemannian manifold $M$ decomposes as a direct sum
\begin{align}
  \label{eqn:Classical_Hodge}
  \Omega^k(M) &\cong \mathcal{H}^k(M) \oplus \mathcal{H}^k(M)^{\perp} \\
   &\cong \mathcal{H}^k(M) \oplus \im(d) \oplus \im(d^*) \, , \nonumber
\end{align}
where we identify $\mathcal{H}^k(M)$ with the space of harmonic $k$-forms, 
and $d$ is the exterior derivative, so that $\im(d)$
is the space of exact $k$-forms and $\im(d^*)$ is the
space of co-exact $k$-forms. We can further identify $\mathcal{H}^k(M)$ with
$\ker(d) \cap \ker(d^*)$, as well as with the kernel of the Laplacian
$\mathcal{L} = dd^* + d^*d$. The various guises of harmonic forms have 
created unifying bridges between disparate areas of mathematics. Applications
of harmonicity vary from studying heat flow~\cite{rosenberg_laplacian_1997,
warner_foundations_1983} to analyzing random walks~\cite{wilmer_markov_2009,
mukherjee_random_2016} and more generally spectral geometry~\cite{kac1966can, chavel1984eigenvalues}, to name a few.

In the 1940s, Eckmann generalized the classical story by replacing the manifold $M$ with
a simplicial complex $X$, the $k$-forms with the simplicial $k$-chains, and the Laplacian with an analogously defined discrete Laplacian \cite{Eckmann1944}.
In this discrete setting, Eckmann found that a theorem analogous to the Hodge decomposition, sometimes called the \textit{discrete Hodge decomposition}, still holds.
That is, the space of $k$-chains can be written as an direct sum
\begin{equation}\label{eq:Eckmann_Hodge}
  C_k(X;\mathbb C) \cong 
  \ker(\mathcal L_k) \oplus \im(\partial) \oplus \im(\partial^*)  \, ,
\end{equation}
where $\mathcal L_k$ here is the discrete Laplacian $\mathcal L_k = \partial\partial^* + \partial^*\partial$.
Besides its theoretical significance, this result furnishes many practical
applications of harmonicity including voting and ranking
theory~\cite{diaconis1989generalization, vigna_spectral_2016, jiang_statistical_2011}, network theory~\cite{adiprasito2018hodge,baker2018hodge}
and spectral graph theory~\cite{Lim2015}.  

One particular use of the 
Hodge decomposition is the following.  
Every chain in $\ker(\mathcal L_k)$ is a cycle, and the natural map
$\ker(\mathcal L_k) \to H_k(X;\mathbb C)$ sending each cycle to its homology
class is an isomorphism.
Inverting this isomorphism lets us select a canonical representative of each homology class.
In other words, the 
Hodge decomposition implies that every homology class has a unique harmonic representative.
Among other implications, this result lets allows us to perform homology computations in a subspace instead of a quotient space, the latter being tricky to implement in practice.


The story we have laid out so far holds just as well if we replace $\mathbb C$ with any field of characteristic 0. 
However, we are often interested in homology and cohomology over more general coefficient rings, including fields of positive characteristic.
This is especially true in applications; computations over finite fields, especially $\mathbb{F}_2$, can be performed more quickly than over $\mathbb{Q}$, and with perfect precision.
One example is the software library Ripser, which is a popular choice for computing persistent (co)homology \cite{Bauer2021}.
Once we learn of the practical uses of the Hodge decomposition and harmonic representatives, it is natural to ask whether similar results hold over other fields.
This leads us to the driving questions of this paper.
\begin{enumerate}
  \item[\textbf{(Q1)}] Does an analog of the discrete Hodge decomposition exist when we replace $\mathbb{C}$ with a finite field, or more generally with an arbitrary field $\mathbb{F}$?
  \item[\textbf{(Q2)}] If not, when does every homology class have a unique harmonic representative?
\end{enumerate}

The first of our primary results, Theorem~\ref{thm:TFAE-list}, provides a list of necessary and sufficient conditions for every homology class in $H_k(X;\F{})$ to have a unique harmonic representative, addressing (Q2).
Subsequently, Theorem~\ref{thm:hodge-decomposition} gives stronger necessary and sufficient conditions for the existence of a Hodge decomposition analogous to Eq.~\eqref{eqn:Classical_Hodge}, namely, if and only if every homology class \textit{and} every cohomology class have a unique harmonic representative. This answers question (Q1).
Sections~\ref{sec:existence-and-uniqueness} and \ref{sec:computation} are spent proving these results and developing methods for computing these unique harmonic representatives.
While our motivating examples arise from CW complexes with finitely many cells in each dimension, in these sections we work in the more general setting of chain complexes of finite-dimensional spaces equipped with standard bases.
For this reason, our results can be applies just as well to the dual chain complex to derive results about cohomology that are ``dual'' to our results about homology.


We then move on to consider (Q1) and (Q2) from a different perspective.
Given a fixed CW complex $X$ with finitely many cells in each dimension, in Section~\ref{sec:rational-methods} we explore for which primes $p$ the cellular homology of $X$ over $\F p$ satisfies the statements of Theorem~\ref{thm:TFAE-list}.
That is, given $X$, we provide sufficient conditions for $p$ under which every cellular homology class of $X$ over $\F p$ has a unique harmonic representative.
We accomplish this by leveraging the cleaner rational case and descending to the finite field.
This method is thwarted if $H_k(X)$ has $p$-torsion, or if $p$ is one of finitely many primes determined by the combinatorial structure of $X$.
Theorem~\ref{thm:fixed-X-variable-p} determines this list of problematic primes in terms of an integer that we call $\Upsilon$ which depends on the combinatorial structure of $X$.

In the special case that $X$ is a CW structure on an orientable surface, the statements of Section~\ref{sec:rational-methods} can be simplified. In Section~\ref{sec:orientable-surfaces} we show that, in this special case, Theorem~\ref{thm:harmonic-coefficient-subring} has a simpler statement in terms of the discrete Laplacian, giving us a tidy condition in Theorem~\ref{cor:orientable-surface-tidy-existence} to ensure the existence and uniqueness of all harmonic representatives on $X$ over $\F p$.

Ultimately the difference between harmonicity in positive and zero characteristic stems
from the fact that vector spaces over fields of positive characteristic cannot be equipped with norms 
or inner products, since fields of positive characteristic cannot be ordered.
The best replacement is a nondegenerate bilinear form, for which there
can exist degenerate subspaces. Degenerate subspaces have nontrivial intersection
with their `orthogonal complement', so a direct sum decomposition as in
Eq.~\eqref{eqn:Classical_Hodge} is generally not possible.
Over a field of characteristic 0, $\ker(\mathcal L_k) = \ker(\partial) \cap
\ker(\partial^*)$. However, this fact depends on the assumption that every
subspace is nondegenerate. Over a field of positive characteristic, we instead
have $\ker(\mathcal L_k) \supseteq \ker(\partial) \cap \ker(\partial^*)$, often
with strict containment. Indeed, over a general field, an element of
$\ker(\mathcal L_k)$ need not even be a cycle (see
Example~\ref{ex:pinched-cylinder-subdivided}). Since this paper is interested
in harmonic representatives of (co)homology classes, we make the choice to
define the harmonic chains to be the space $\ker(\partial) \cap
\ker(\partial^*)$, i.e. the space of chains that are both cycles and cocycles.
We emphasize that, over characteristic 0, this is equivalent to the usual
definition that instead chooses $\ker(\mathcal L_k)$, but in general the two
spaces are distinct.

\subsection{Related Work.}

In~\cite{Ebli2019}, Ebli and Spreemann utilize the discrete laplacian over $\R$
to develop a homology-sensitive method to cluster the simplices of a
simplicial complex. Among other uses, this process produces visualizations of
homology classes similar to those in Section~\ref{sec:experiments}.
A generalization of the discrete Laplacian designed for the setting of persistent homology is introduced and explored in \cite{Wang2020} and \cite{memoli2021}. In the
recent preprint~\cite{basu2021}, Basu and Cox apply the theory of harmonic
representatives to persistent homology to associate concrete
subspaces of chain spaces to each bar of the persistence barcode in a stable
way. Kali\v{s}nik et. al.~\cite{kalivsnik2019higher} explore related methods for
creating a homologically persistent skeleton using higher spanning trees,
and encode important topological features of data using these subcomplexes.

\section{Background and Main Results}\label{sec:background-and-results}
In this section, we briefly recall some facts regarding linear algebra over
arbitrary fields, homology, cohomology, and CW complexes that will be used throughout.
With this background in place, we then state our main results. 

\subsection*{Linear algebra over arbitrary fields}
We remind the reader of some basic facts about vector spaces over fields of
positive characteristic. A standard reference for this material is given
in~\cite{taylor1992geometry}.

A \emph{symmetric bilinear form} on a vector space $V$ over a field $\bF$ is a map
$b\colon V \times V \to \bF$ such that
\begin{gather*}
  b(v,w) = b(w,v) \, , \\
  b(u+v,w) = b(u,w) + b(v,w)\, , \mbox{ and} \\
  b(\lambda v,w) = \lambda b(v,w) \, ,
\end{gather*}
for any $u, v, w \in V$ and $\lambda \in \bF$. Given a symmetric bilinear form
$b$ on $V$ and a subset $U \subseteq V$, we define the \emph{orthogonal 
complement} of $U$ by
\begin{equation*}
  U^{\perp} = \{ v \in V \, | \, b(u,v) = 0 \mbox{ for all } u \in U \} \,.
\end{equation*}
The symmetric bilinear form $b$ is said to be \emph{non-degenerate} if
$V^{\perp} = 0$. If $U$ is a subspace of $V$, then $b|_{U} \colon U \times U \to F$ is again a 
symmetric bilinear form. A subspace $U$ is \emph{non-degenerate} if 
$b|_{U \times U}$ is non-degenerate, i.e., $U \cap U^{\perp} = 0$.

\begin{example}\label{ex:standard-inner-product}
  Let $V$ be an $n$-dimensional vector space over the real numbers $\bF = \R$ with 
  basis $\{e_1, e_2, \ldots, e_n\}$. Then 
  $b(e_i,e_j) = \delta_{ij}$, given by the Kronecker delta function,
  defines a non-degenerate symmetric bilinear form, the standard 
  inner product. Every subspace of $V$ is non-degenerate.

  Let $W$ be an $m$-dimensional vector space over the field with $p$
  elements $\bF = \F{p}$ with basis $\{f_1, f_2, \ldots, f_m\}$. Then 
  $b(f_i,f_j) = \delta_{ij}$ defines a non-degenerate symmetric bilinear
  form. In this case, $b$ does not define an inner product and there are
  numerous degenerate subspaces of $W$, e.g., the one-dimensional subspace
  spanned by $f_1 + \cdots + f_p$ for $ p \leq m$.
\end{example}

  


\subsection{Homology and Cohomology}
Let $R$ be a ring.
A \emph{chain complex} $C_{\bullet}$ over $R$ is a sequence of $R$-modules $\{C_k\}_{k\in \mathbb Z}$ called the \emph{chain modules}, along with homomorphisms $\partial_k:C_k \to C_{k-1}$, called the \emph{differentials}, satisfying $\partial_{k-1}\partial_k = 0$.
The $k$-cycles $Z_k \subseteq C_k$ are given by the kernels of the
differentials
$
  Z_k = \ker \partial_k \, .
$
The $k$-boundaries $B_k \subseteq C_k$ are given by the images of the
differentials
$
  B_k = \im \partial_{k+1} \, .
$
Finally, since the condition $\partial_{k-1}\partial_{k} = 0$ implies $B_k \subseteq Z_k$, we can define the
homology of the chain complex in degree $k$ to be the quotient
$
  H_k = Z_k/B_k \, .
$
We use the notation $[z] \in H_k$ to denote the equivalence class of $z \in Z_k$, and say that $z$ is a \emph{representative} of $[z]$.

The associated {\em cochain complex} is the $R$-linear dual of the
chain complex. Precisely it is the sequence of $R$-modules $C^k := \hom_R(C_k,R)$,
the $R$-linear dual of $C_k$, together with codifferentials given by the duals
of the original differentials $\partial_k^* : C^{k-1} \to C^k$. The condition
$\partial_{k-1} \partial_{k} = 0$ implies $\partial^*_{k}\partial^*_{k-1} = 0$.
Analogous to cycles and boundaries, the $k$-cocycles $Z^k \subseteq C^k$ are the cochains in the kernel of the codifferential
$Z^k = \ker(\partial^*_{k+1}),$
the $k$-coboundaries $B^k \subseteq C^k$ are those in the images
$B^k = \im(\partial^*_k).$
and the cohomology of the cochain complex is given by the quotient
$H^k = Z^k/B^k.$

\begin{remark}\label{rem:identification}
  We are interested in the case where the underlying ring of the chain complex is a field $\bF$, and each chain module $C_k$ is a finite-dimensional vector space with a chosen basis.
  This is the case when $C_{\bullet}$ arises as a simplicial or cellular chain complex as discussed below in section \ref{subsec:cellular homology}
  The chosen basis, say $\{ \alpha_i\}$ on $C_k$, defines a symmetric bilinear form $b_k$ on $C_k$ by setting $b_k(\alpha_i, \alpha_j) = \delta_{ij}$, the Kronecker pairing.
  The form $b_k$ gives rise to a canonical isomorphism between $C_k$ and its dual $C^k$, letting us identify the two and blur the distinction between a chain $c$ and its dual $c^*$.
  This identification is common in the literature, allowing one to
  think of the codifferentials $\partial^*_k$ as linear maps $C_{k-1} \to C_k$, and think of $Z^k$ and $B^k$ as subspaces of $C_k$.
\end{remark}

Given the identification in Remark~\ref{rem:identification}, we state our notion of $\mathbb{F}$-harmonicity.
\begin{definition}
  \label{defn:harmonic}
  A chain $c\in C_k$ is called \emph{$\F{}$-harmonic} if $\partial_k c = \partial_{k+1}^* c = 0$. Equivalently, a chain is $\F{}$-harmonic if it lies in $Z_k \cap Z^k$. When the underlying field is clear, we may simply say $c$ is harmonic.
\end{definition}

Oftentimes, authors working over fields of characteristic 0 will define a chain
to be harmonic if it falls within the kernel of the Laplacian $\mathcal L_k :=
\partial_{k+1}\partial_{k+1}^* + \partial_k^* \partial_k$.  Over such a field, this definition
is equivalent to Definition~\ref{defn:harmonic}, though the two may diverge over a field of positive
characteristic. In general, $\ker \mathcal{L}_k \supseteq Z_k \cap Z^k$ (see Example~\ref{ex:pinched-cylinder-subdivided}).

\subsection*{CW Complexes}
Let $X$ denote a CW complex. We denote the $k$-skeleton
of $X$ by $X^{(k)}$ and the set of $k$-cells is denoted $X_k$. Recall that a CW
complex is given by iteratively attaching $k$-cells
$e_{\alpha}^k$ as $\alpha$ varies, where each $e_{\alpha}^k$ is homeomorphic to
an open Euclidean $k$-disk.  Inducting on the degree $k$, the $k$-skeleton is
formed from the $(k-1)$-skeleton via the attaching maps for each $k$-cell. The
attaching map of a cell $e^k_{\alpha}$ is of the form \[ \varphi_{\alpha}
  \colon S_{\alpha}^{k-1} \to X^{(k-1)} \, , \] where $\delta e^k_{\alpha} =
  S^{k-1}_\alpha$ denotes the boundary of $e^k_{\alpha}$.
Additionally, we call a CW complex \textit{degree-wise finite} if it has finitely many cells of each dimension.

All of the topological spaces we are interested in, e.g., simplicial 
and cubical complexes, admit the structure of a CW complex, including
topological graphs.

\subsection*{Cellular Homology}\label{subsec:cellular homology}
Given a CW complex $X$ and a ring $R$, there is a chain complex $C_\bullet(X;R)$ called the \emph{cellular chain complex}. 
Each chain module $C_k(X;R)$ is a free $R$-module with generating set given by the set of $k$-cells of $X$,
so a typical $k$-chain is given by an $R$-linear combination of $k$-cells.
The differentials are determined by the attaching
maps of the cell complex via
\[
  \partial_{k+1}(\alpha) = \sum_{a \in X_k} \langle \delta \alpha, a \rangle a \, ,
\]
for each $(k+1)$-cell $\alpha$, where $\langle - , - \rangle$ denotes
incidence number.

In the case that $R$ is a field and $X$ is degree-wise finite, each cellular chain space $C_k(X;R)$ is finite dimensional with a canonical basis given by the $k$-dimensional cells of $X$.
This case is our motivating example.

\subsection*{Main Results}

We state the first of our main theorems in the general setting of chain complexes of finite
dimensional vector spaces over an arbitrary field with chosen bases, though our motivating
example is that of a cellular chain complex $C_\bullet(X;\bF)$ of a degree-wise finite CW complex $X$.
\begin{bigthm}\label{thm:TFAE-list}
    Let $C_\bullet$ be a chain complex of finite dimensional vector spaces
    over a field $\bF$, where each $C_k$ is equipped with a chosen basis and the induced symmetric bilinear form. The following are equivalent.
    \begin{enumerate}
        \item \label{it:A-ExistenceUniqueness}
        (Existence and uniqueness of harmonic representatives) Every homology class in $H_k$ has a unique harmonic representative.
        \item \label{it:A-Uniqueness}
        (Uniqueness of harmonic representatives) For every homology class in $H_k$ with a harmonic representative, that representative is unique.
        \item \label{it:A-UniquenessAt0}
        The only harmonic representative of the trivial homology class $[0]\in H_k$ is 0.
        \item \label{it:A-HomologyIsomorphism}
        $Z_k \cap Z^k \cong H_k$ via the quotient map $z\mapsto [z]$.
        \item \label{it:A-TrivialIntersection}
        $B_k \cap Z^k = 0$.
        \item \label{it:A-OrthogonalSum}
        $B_k \oplus Z^k = C_k$.
        \item \label{it:A-KernelOfComposition}
		$\ker(\partial_{k}^* \partial_{k}) = \ker(\partial_{k})$.
        \item \label{it:A-ImageOfComposition}
		$\im(\partial_{k+1}^* \partial_{k+1}) = \im(\partial_{k+1}^*)$.
        \item \label{it:A-PseudoinverseExistence}
        The Moore-Penrose pseudoinverse $\pi^\dagger$ of the projection map $\pi:C_k \to C_k/B_k$ exists.
    \end{enumerate}
    Furthermore, if these statements hold, then the pseudoinverse $\pi^\dagger$ of (\ref{it:A-PseudoinverseExistence}) takes every homology class to its unique harmonic representative.
\end{bigthm}

Theorem~\ref{thm:TFAE-list} provides several possible answers to (Q2).

We highlight some notable aspects of Theorem~\ref{thm:TFAE-list}. Statements
(\ref{it:A-ExistenceUniqueness}) through (\ref{it:A-UniquenessAt0}) regard
existence and uniqueness of harmonic representatives.
While uniqueness of harmonic representatives does imply existence, the converse does
not hold, as Example~\ref{ex:pinched-cylinder-subdivided} will demonstrate.
Statements (\ref{it:A-TrivialIntersection}) and (\ref{it:A-OrthogonalSum})
relate to the linear algebraic properties of the chain complex. Note
that $B_k^\perp = \ker(\partial_k) = \im(\partial_k^*) = Z^k$, so $B_k$ and
$Z^k$ are orthogonal complements. If the underlying field has characteristic 0,
these spaces are indeed linear complements and statement
(\ref{it:A-TrivialIntersection}) always holds. Hence the statements of
Theorem~A hold over any characteristic 0 field. 
Finally, statement (\ref{it:A-PseudoinverseExistence})
provides an explicit method for computing the unique harmonic representative of
a homology class.
This is our primary method of computing representatives.
These computations, as well as pseudoinverses over arbitrary fields, are discussed in Section~\ref{sec:computation}.

One advantage of stating Theorem~\ref{thm:TFAE-list} at the level of chain complexes
is that it can be applied just as well to a cochain complex, such as the dual complex $C^\bullet$, so long as we take care with the indexing.
Recall that a cochain complex is roughly the same as a chain complex, except that the indexing increases as we apply the boundary map instead of decreasing.
Since the indexing is completely incidental to the above statements, we can apply Theorem~\ref{thm:TFAE-list} to $C^\bullet$ and get a ``dual'' theorem.
For example, when
Theorem~\ref{thm:TFAE-list} is applied to a cochain complex, statement~(\ref{it:A-TrivialIntersection}) becomes ``$B^k \cap Z_k = 0$,'' and statements~(\ref{it:A-ExistenceUniqueness}) through (\ref{it:A-HomologyIsomorphism}) become statements about harmonic representatives of \emph{co}homology classes.
When the statements of Theorem~\ref{thm:TFAE-list} apply
to both a chain complex $C_\bullet$ and its dual complex $C^\bullet$, these
statements will give us that the harmonic forms are exactly $\ker \mathcal L$,
as in the classical case.
\begin{definition} If the statements of Theorem~\ref{thm:TFAE-list} hold for a cochain complex $C_\bullet$ at degree $k$, we say that $C_\bullet$ is \emph{homologically harmonic} in degree $k$. If the statements of Theorem~\ref{thm:TFAE-list} hold for a cochain complex $C^\bullet$ at degree $k$, we say that $C^\bullet$ is \emph{cohomologically harmonic}.
We say a degree-wise finite CW complex $X$ is \textit{homologically harmonic in degree $k$ over $\bF$} when the cellular chain complex $C_\bullet(X;\bF)$ is homologically harmonic, and we say $X$ is \emph{cohomologically harmonic in degree $k$ over $\bF$} if the dual complex $C^\bullet$ is cohomologically harmonic.
\end{definition}

\begin{bigthm}[Hodge Decomposition over Arbitrary Fields]\label{thm:hodge-decomposition}
  Let $C_\bullet$ be a chain complex of finite-dimensional vector spaces with chosen bases, and let $\mathcal L_k:C_k\to C_k$ be the discrete Laplacian $\partial\partial^* + \partial^*\partial$.
  The vector space $C_k$ decomposes into a direct sum of subspaces
  \[
    C_k = \ker \mathcal{L}_k \oplus B_k \oplus B^k
  \]
  if and only if both $C_\bullet$ is homologically harmonic in degree $k$, and the dual complex $C^\bullet$ is cohomologically harmonic in degree $k$.
  
  In this case, the subspaces in the decomposition are orthogonal with respect to the symmetric bilinear form, and the component $\ker \mathcal L_k$ is equal to $Z_k \cap Z^k$ . Furthermore, 
  $\ker \mathcal {L}_k$ is isomorphic to both $H_k$ and $H^k$ through the natural maps sending (co)cycles to their respective (co)homology classes.
\end{bigthm}

While Theorem~\ref{thm:hodge-decomposition} provides an answer to (Q1) in the general setting of chain complexes, we are most interested in the case when $C_\bullet = C_\bullet(X;\F{})$ arises from some degree-wise finite CW complex $X$. In this case, Theorem~\ref{thm:hodge-decomposition} provides necessary and sufficient conditions for a decomposition
\[C_k(X;\F{}) = \ker(\mathcal L_k) \oplus B_k(X;\F{}) \oplus B^k(X;\F{}).\]








In Section~\ref{sec:rational-methods}, we focus on the special case of cellular
homology of a degree-wise finite CW complex $X$ over finite fields. In this
special case, we can ask for which primes $p$ the complex $X$ is homologically
harmonic over $\F{p}$. Our approach is to take the matrix sending each rational
cycle to its unique harmonic representative and use it to construct a similar
map over a finite field $\F{p}$.  Since rational homology cannot detect
torsion, this approach requires that $H_k(X)$ have no $p$-torsion. The
feasibility of this construction also depends on the entries of the matrix,
which we show fall within a certain subring of $\mathbb{Q}$ determined by
the combinatorial structure of $X$. In the process we prove the following
result, which we believe to be of independent interest. For an integer $N$, let
$\Z[N^{-1}]$ denote the localization $S^{-1}\Z$, where
$S = \{1, N, N^2, N^3, \ldots\}$.

\begin{bigthm}\label{thm:harmonic-coefficient-subring}
  For a finite CW complex $X$, there exists an integer
  $\Upsilon = \Upsilon(X,k)$ depending on the combinatorial structure of $X$ so that, for any integral homology class in degree $k$, its rational harmonic representative has coefficients in $\Z[\Upsilon^{-1}]$.
\end{bigthm}

This result is justified by Theorem~\ref{thm:projection-entries}, which is stated after we give a rigorous definition of $\Upsilon$ in Definition~\ref{def:constant-D}. The primary result of Sections~\ref{sec:rational-methods} and \ref{sec:orientable-surfaces} is the following sufficient condition for the homological harmonicity of $X$.

\begin{bigthm}\label{thm:fixed-X-variable-p}
  If $p \nmid \Upsilon(X,k)$, and $H_k(X)$ has no $p$-torsion, then $X$ is homologically harmonic over $\F{p}$ in degree $k$. That is, every homology class in $H_k(X;\F{p})$ has a unique harmonic representative.

  Furthermore, if $k=1$ and $X$ is a CW structure on a connected orientable surface, $\Upsilon$ is given by
  \[\Upsilon = \det \widehat{\mathcal L}_1 / |X_{2}| \, ,\]
  where $\widehat{\mathcal L}_1 = \mathcal{L}_1|_{B_1(X)}: B_1(X;\F p) 
  \stackrel{\cong}{\longrightarrow} B_1(X;\F p)$ is
  the Laplacian restricted to the $1$-boundaries.
\end{bigthm}
Since only finitely many primes divide $\Upsilon$, and $H_k(X)$ has nontrivial $p$-torsion for only finitely many primes $p$, we get the following corollary.
\begin{corollary}
  Given a degree-wise finite CW complex $X$ and $k\geq 0$, $X$ is homologically harmonic in degree $k$ over $\F{p}$ for all but finitely many primes $p$.
\end{corollary}

Finally, in Section~\ref{sec:experiments}, we put the above theorems to work by computing and plotting some examples of harmonic representatives over various fields.
Our experiments will include examples on orientable surfaces and Vietoris Rips complexes.

\section{Existence and Uniqueness of Harmonic Representatives}\label{sec:existence-and-uniqueness}
In this section, we explore the properties of a chain complex that affect the existence and uniqueness of harmonic representatives.
For a fixed CW complex,
a given homology class may have
multiple harmonic representatives or none at all depending on the characteristic of the
coefficient field.
More generally, a homology class of a chain complex may fail to exhibit a harmonic representative, or have many.
We begin by
considering examples that demonstrate failures of existence and uniqueness.

\begin{example}\label{ex:pinched-cylinder}
  Consider a CW complex $X$ with a single 0-cell, two 1-cells denoted $a$ and $b$,
  and a single 2-cell denoted $E$ with an attaching map that travels once
  around $a$, then once around $b$. The resulting CW complex, depicted in
  Figure~\ref{fig:existenceExample}, can be thought of as a cylinder in which a longitudinal line
  segment has been collapsed. We will consider
  the homology of this complex with $\F{2}$ coefficients.

  Since $X$ is homotopy equivalent to a circle, it has exactly one nontrivial
  homology class in degree 1, namely the class $[a] = \{a,b\}$. However,
  neither $a$ nor $b$ is a cocycle, since $\partial_2^*(a) = \partial_2^*(b) = E$.
  Therefore the homology class $[a]$ has no $\F{2}$-harmonic representative.

  In addition, both cycles in the trivial homology class $[0] = \{0,a+b\}$ are
  cocycles since $\partial_2^*(a+b) = 2E = 0$, 
  so $[0]$ has two $\F{2}$-harmonic representatives.
\end{example}

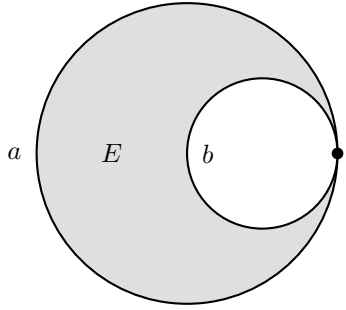
\begin{figure}
  \begin{tikzpicture}
    \draw[fill=gray!25!white,thick] (2,2) circle (2cm);
    \draw[fill=white,thick] (3,2) circle (1cm);
    \node at (4,2) [circle,draw=black,fill=black,inner sep=0.5mm] {};
    \node at (2,2) [circle,draw=none,fill=none,inner sep=0.5mm, label=right:{$b$}] {};
    \node at (0,2) [circle,draw=none,fill=none,inner sep=0.5mm, label=left:{$a$}] {};
    \node at (1,2) [circle,draw=none,fill=none,inner sep=0.5mm, label=center:{$E$}] {};
  \end{tikzpicture}
  \caption{ \label{fig:existenceExample}
  A CW complex that fails both existence and uniqueness of $\F2$-harmonic representatives.}
\end{figure}

\begin{example}\label{ex:pinched-cylinder-subdivided}
  Take the CW complex $X$ from Example~\ref{ex:pinched-cylinder} and subdivide
  the 2-cell $E$ with an edge $e$ into a pair of 2-cells $G$ and $H$ as pictured in
  Figure~\ref{fig:uniquenessExample}; call this new complex $X'$. We will
  again consider the homology of $X'$ with $\F2$ coefficients.
  As before, $X'$ has one nontrivial
  homology class in degree one. There are four cycles in the homology class
  of $b+b'$, and 
  computing the coboundary of each representative reveals that
  \begin{gather*}
    \partial_2^*(a+a') = \partial_2^*(b+b') = G + H\, , \mbox{ and}\\
    \partial_2^*(a+b'+e) = \partial_2^*(a'+b+e) =  0.
  \end{gather*}
  Hence the class of $b+b'$ has two $\F2$-harmonic representatives, namely
  $a+b'+e$ and $a'+b+e$.

  We can perform a similar computation beginning with the trivial cycle $0$.
  In this case, we find that the class of $0$ also has two $\F2$-harmonic
  representatives, given by $0$ and $a+a'+b+b'$. 
  Additionally, the 1-chain $e$ lies in the kernel of
  the laplacian $\mathcal L_1$, since $\partial_2\partial_2^*e =
  \partial_1^*\partial_1 e = a+b+a'+b'$. This further demonstrates that in general elements of
  $\ker \mathcal L_k$ need not be cycles, and that the containment $Z_k \cap
  Z^k \subseteq \mathcal L_k$ can be strict.

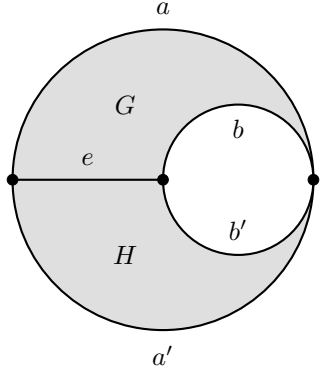
\begin{figure}
  \begin{tikzpicture}
    \draw[fill=gray!25!white,thick] (2,2) circle (2cm);
    \draw[fill=white,thick] (3,2) circle (1cm);
    \draw[thick] (0,2) -- (2,2);
    \node at (4,2) [circle,draw=black,fill=black,inner sep=0.5mm] {};
    \node at (2,2) [circle,draw=black,fill=black,inner sep=0.5mm] {};
    \node at (0,2) [circle,draw=black,fill=black,inner sep=0.5mm] {};
    \node at (3,3) [circle,draw=none,fill=none,inner sep=0.5mm, label=below:{$b$}] {};
    \node at (3,1) [circle,draw=none,fill=none,inner sep=0.5mm, label=above:{$b'$}] {};
    \node at (2,4) [circle,draw=none,fill=none,inner sep=0.5mm, label=above:{$a$}] {};
    \node at (2,0) [circle,draw=none,fill=none,inner sep=0.5mm, label=below:{$a'$}] {};
    \node at (1,2) [circle,draw=none,fill=none,inner sep=0.5mm, label=above:{$e$}] {};
    \node at (1.5,3) [circle,draw=none,fill=none,inner sep=0.5mm, label=center:{$G$}] {};
    \node at (1.5,1) [circle,draw=none,fill=none,inner sep=0.5mm, label=center:{$H$}] {};
  \end{tikzpicture}
  \caption{ \label{fig:uniquenessExample} A CW complex which fails uniqueness but not existence; every homology class has a non-unique $\F2$-harmonic representative.}
\end{figure}

\end{example}

\subsection{Uniqueness}
Given an $\bF$-harmonic representative for a homology class, we can measure
its failure to be unique by characterizing the set of all $\bF$-harmonic
representatives for its homology class.
Suppresing the field $\bF$ from the notation, let $\Har(z)$ be the set of all $\bF$-harmonic
representatives of $z \in H_k$.

\begin{proposition}\label{prop:uniqueness-measurement}
  When $\Har(z)$ is non-empty, it forms an additive torsor over $B_k \cap Z^k$.
  Equivalently, if $\Har(z) \neq \emptyset$, then $\Har(z) = h +
  (B_k \cap Z^k)$ for any $h\in \Har(z)$. In particular, the $\bF$-harmonic representatives of the 
  trivial homology class coincide with $B_k \cap Z^k$.
\end{proposition}

\begin{proof}
  Suppose $h,h' \in \Har(z)$. By definition, $h$, $h' \in Z^k$ and thus,
  $h'-h \in Z^k$. Since $[h]=[h'] = z$, $h$ and $h'$ must
  differ by a boundary. Hence $h'-h \in B_k\cap Z^k$, showing $h'\in
  h+(B_k\cap Z^k)$. For the reverse inclusion, let $h \in \Har(z)$ and
  pick an arbitrary element $h+b \in h+(B_k\cap Z^k)$. Since both $h$ and
  $b$ are cocycles, so is their sum $h+b$. Also, $[h+b] = [h] = z$, so
  we conclude $h+b$ is a harmonic representative of $z$. The second
  statement follows immediately.
\end{proof}

In general, $B_k = (Z^k)^\perp$ with respect to any symmetric non-degenerate
bilinear form on $C_k$, 
since $B_k = \im(\partial_{k+1}) = \ker(\partial^*_{k+1})^\perp = (Z^{k})^\perp$. This means
that $B_k \cap Z^k = 0$ if and only if $B_k$ (equivalently, $Z^k$) is a
nondegenerate subspace of $C_k$. Combined with
Proposition~\ref{prop:uniqueness-measurement}, this gives the following.

\begin{corollary}\label{cor:uniqueness-condition}
Every harmonic representative in $C_k$ is unique if and only if $B_k$
(equivalently $Z^k$) is a non-degenerate subspace of $C_k$ with respect to the bilinear form on $C_k$.
\end{corollary}

\begin{remark}\label{rmk:char-0-uniqueness}
  If $\bF$ is a field of characteristic 0, every subspace of $C_k$ is
  non-degenerate. Corollary~\ref{cor:uniqueness-condition} then implies
  that harmonic representatives are always unique over a field of
  characteristic 0.
\end{remark}

\subsection{Existence}
We turn our attention upon the extent to which harmonic representatives exist.

  \begin{theorem}\label{thm:existence-measurement}
    Let $Q_k$ denote the subspace of all degree $k$-homology classes that have at least one harmonic representative. Then
    \[
      Q_k = \frac{Z_k \cap (Z^k + B_k)}{B_k} \, .
    \]
    The quotient $H_k/Q_k$ is isomorphic to $\frac{Z^k+Z_k}{Z^k+B_k}$, and in particular, the codimension of $Q_k$ in $H_k$ is $\dim(Z^k+Z_k)-\dim(Z^k+B_k)$.
  \end{theorem}

\begin{proof}
  First we show that $Q_k=\frac{Z_k \cap (Z^k + B_k)}{B_k}$ by dual
  containment. Let $z + B_k\in Q_k$. Write $z
  + B_k = h + B_k$ for some harmonic chain $h$. Since $h\in Z_k\cap Z^k
  \subseteq Z_k\cap (Z^k + B_k)$, we can conclude $h + B_k \in \frac{Z_k
  \cap (Z^k + B_k)}{B_k}$. For the reverse inclusion, let $z +B_k \in
  \frac{Z_k \cap (Z^k + B_k)}{B_k}$, so that $z \in Z_k \cap (Z^k + B_k)$.
  Since $z\in Z_k$, $z + B_k \in H_k$. Furthermore, $z = z' + b$ for some
  $z'\in Z^k$, $b\in B_k$. Since both $b$ and $z$ are cycles, $z'$ is a
  cycle as well. Hence $z'$ is a representative of the homology class of
  $z$ that is both a cycle and a cocycle. This implies $z + B_k \in Q_k$.

  Now that we have shown $Q_k=\frac{Z_k \cap (Z^k + B_k)}{B_k}$, we have
  \[
    H_k/Q_k = \frac{\frac{Z_k}{B_k}}{\frac{Z_k \cap (Z^k + B_k)}{B_k}} \cong \frac{Z_k}{Z_k \cap (Z^k + B_k)} \, .
  \]
  By the second isomorphism theorem,
  \[
    \frac{Z_k}{Z_k \cap (Z^k + B_k)} \cong \frac{Z^k + B_k + Z_k}{Z^k + B_k} = \frac{Z^k + Z_k}{Z^k + B_k} \, . \qedhere
  \]
\end{proof}

\begin{remark}
  As discussed in Remark~\ref{rmk:char-0-uniqueness}, over a field of
  characteristic 0, $B_k$ and $Z^k$ are orthogonal complements in $C_k$.
  Therefore $Z^k+B_k = C_k$, and $\frac{Z^k+Z_k}{Z^k+B_k} = 0$.
  Theorem~\ref{thm:existence-measurement} then implies that every homology class has
  a harmonic representative. In addition, by
  Remark~\ref{rmk:char-0-uniqueness}, that harmonic representative is unique.
\end{remark}

\section{Computation of Harmonic Representatives}\label{sec:computation}
In this section, we focus on computing harmonic representatives by introducing the 
Moore-Penrose pseudoinverse. We then use the results of
Sections~\ref{sec:existence-and-uniqueness} and \ref{sec:computation} to prove
Theorem~\ref{thm:TFAE-list}, as well as Theorem~\ref{thm:hodge-decomposition}, our generalization of the
discrete Hodge decomposition.

\subsection{Pseudoinverses}
\begin{definition}
	Given an $m\times n$ matrix $A$ with entries in $\bF$, the (Moore-Penrose)
	\textit{pseudoinverse} of $A$, denoted $A^\dagger$, is an $n\times m$ matrix
	satisfying the following:
	\begin{enumerate}
		\item $AA^\dagger A = A^\dagger$.
		\item $A^\dagger A A^\dagger = A^\dagger$.
		\item $AA^\dagger$ is Hermitian.
		\item $A^\dagger A$ is Hermitian.
	\end{enumerate}
\end{definition}


The (Moore-Penrose) pseudoinverse generalizes the usual matrix inverse to
matrices that fail to be invertible. A preferred solution to the system $Ax =
y$ is given by the pseudoinverse: $A^\dagger y$. If $A$ is invertible, then
$A^\dagger = A^{-1}$ and the preferred solution is the unique solution. If $Ax
= y$ has multiple solutions, $A^\dagger y$ is the solution orthogonal to
$\ker(A)$, which is also the solution of minimal norm if $A$ is over a field of
characteristic 0. If $Ax = y$ has no solution, then $A^\dagger y$ is a solution
to the system $Ax = \hat y$ where $\hat y$ is the orthogonal projection of $y$
onto $\im(A)$. Over a field of characteristic 0, this is equivalent to saying
that $A^\dagger$ minimizes the norm of $y-A(A^\dagger y)$.

The following existence and uniqueness theorems are due to Pearl~\cite[Theorem 1]{Pearl1968} and Penrose\footnote{While Penrose was considering matrices with complex entries, his uniqueness proof works over any field.}~\cite[Theorem 1]{Penrose1955}.
\begin{theorem}\label{thm:existence-uniqueness-of-pseudoinverses}
	Let $A$ be a matrix with entries in a field.
	\begin{enumerate}
		\item (Pearl) A pseudoinverse of $A$ exists if and only if $\rank(AA^*) = \rank(A) = \rank(A^*A)$.
		\item (Penrose) If a pseudoinverse of $A$ does exist, then it is unique.
	\end{enumerate}
\end{theorem}
The uniqueness result justifies the use of the notation $A^\dagger$ to denote \emph{the} pseudoinverse of $A$.

\begin{remark}
	If $A$ is a matrix over a field of characteristic 0, then $\rank(AA^*) = \rank(A^*A) = \rank(A)$ always holds, so $A^\dagger$ always exists.
\end{remark}

\begin{remark}\label{rmk:full-rank-pseudoinverse}
	If $A$ is full rank, then $A^*A$ is automatically invertible, so by Theorem~\ref{thm:existence-uniqueness-of-pseudoinverses}, $A^\dagger$ exists if and only if $(AA^*)\inv$ exists. If this is the case, then it can be shown that $A^\dagger=A^*(AA^*)\inv$ by directly verifying the pseudoinverse axioms.
\end{remark}

Given a linear map $\phi:V\to W$ between finite dimensional
vector spaces, the pseudoinverse is traditionally only defined with respect to chosen bases
for $V$ and $W$. In the special case that $\phi$ is surjective, however, we
need only specify a preferred basis for $V$ in order for a pseudoinverse of
$\phi$ to be well-defined.


\begin{definition}
  Let $\phi:\bF^n \to W$ be a linear map and $A$ a matrix representing $\phi$ with respect to a basis $B\subseteq W$. If $A^\dagger$ exists, then the pseudoinverse of $\phi$ is defined by
	$\phi^\dagger(w) = A^\dagger w$.
\end{definition}

We show that this definition does not depend on the
choice of basis $B$. Let $B'$ be another ordered basis for $W$. Define
$\alpha_B:\bF^m \to W$ to be the linear isomorphism sending the standard basis
to $B$, and define $\alpha_{B'}$ similarly. Let $A$ and $A'$ be matrices
representing $\phi$ with respect to $B$ and $B'$ respectively, so that  
\[
\begin{tikzcd}
	\bF^n \ar[r,"A"] \ar[d,"A'" '] \ar[dr,"\phi"] & \bF^m \ar[d,"\alpha _B"]\\
	\bF^m \ar[r,"\alpha_{B'}"'] & W
\end{tikzcd}
\]
commutes. We will show that
\begin{equation}
\begin{tikzcd}
	\bF^n & \bF^m \ar[l,"A^\dagger" '] \\
	\bF^m \ar[u,"A'^\dagger"]   & W \ar[l,"\alpha_{B'}\inv"] \ar[u,"\alpha_B\inv"']
\end{tikzcd}
\label{eqn:pseduo_basis}
\end{equation}
commutes as well. The two paths from $W$ to $\bF^n$ in
Diagram~\eqref{eqn:pseduo_basis} represent the two possible definitions of
$\phi^\dagger$, each with respect to a different basis for $W$, so showing 
Diagram~\eqref{eqn:pseduo_basis} commutes will suffice. By
Remark~\ref{rmk:full-rank-pseudoinverse}, $A^\dagger = A^*(AA^*)\inv$, so
\begin{align*}
	A^\dagger \alpha_B \inv
	= & A^*(AA^*)\inv \alpha_B \inv.
\end{align*}
By the first diagram, $A = \alpha_B\inv\alpha_{B'} A'$. Substituting this into the above expression and simplifying yields
$A'^\dagger \alpha_{B'}\inv$, showing that the diagram commutes.

The following result follows immediately from Remark~\ref{rmk:full-rank-pseudoinverse}.

\begin{proposition}\label{prop:pseudoinverse-of-surjective-map-existence-and-form}
	If $\phi:\mathbb{F}^n\to W$ is a surjective linear transformation, then
	$\phi^\dagger$ exists if and only if $\phi\phi^*$ is invertible. In this
	case, $\phi^\dagger = \phi^*(\phi\phi^*)\inv$.
\end{proposition}

%

\subsection{Computing the Harmonic Representative}

In the special case that each homology class has a unique harmonic representative, that representative can be computed using a pseudoinverse.

\begin{lemma}\label{lemma:image-pi-dual}
	Let $\pi:C_k\to C_k/B_k$ be the canonical projection map. Then $\im(\pi^*)=Z^k$.
\end{lemma}

\begin{proof}
  By definition of $\pi$, we find $B_k = \im(\partial_{k+1}) = \ker(\pi)$. Taking orthogonal complements,
	\begin{equation*}
    \ker(\partial_{k+1}^*) = \im(\partial_{k+1})^\perp = \ker(\pi)^\perp = \im(\pi^*). \qedhere
	\end{equation*}
\end{proof}

\begin{theorem}\label{thm:pseudoinverse-usefulness}
	Suppose the canonical projection $\pi:C_k\to C_k/B_k$ has pseudoinverse $\pi^\dagger$. For any $h \in H_k$, $\pi^\dagger(h)$ is a harmonic representative for $h$.
\end{theorem}

\begin{proof}
  By Proposition~\ref{prop:pseudoinverse-of-surjective-map-existence-and-form},
  $\pi^\dagger = \pi^*(\pi\pi^*)^{-1}$, and in particular this shows that
  $\pi^\dagger$ is a right inverse of $\pi$. Hence $h=\pi\pi^\dagger(h) =
  \left[\pi^\dagger(h)\right]$, showing that $\pi^\dagger(h)$ is a
  representative of $h$.

  In addition, since $\pi^\dagger = \pi^*(\pi\pi^*)^{-1}$, and
  $(\pi\pi^*)^{-1}$ is full rank, we can see $\im(\pi^\dagger) = \im(\pi^*)$,
  and by Lemma~\ref{lemma:image-pi-dual}, $\im(\pi^*) = Z^k$. Therefore
  $\pi^\dagger(h)$ is indeed a cocycle.
\end{proof}

Theorem~\ref{thm:pseudoinverse-usefulness} allows for the computation of
harmonic representatives when $\pi^\dagger$ exists but, as stated in
Theorem~\ref{thm:existence-uniqueness-of-pseudoinverses}, this need not always be the
case.

\begin{lemma}\label{lemma:pseudoinverse-existence-condition}
	The projection $\pi:C_k\to C_k/B_k$ has a pseudoinverse $\pi^\dagger$ if and only if every harmonic representative is unique.
\end{lemma}

\begin{proof}
	As discussed in Remark~\ref{rmk:full-rank-pseudoinverse}, $\pi^\dagger$ exists if and only if $\pi\pi^*$ is invertible, and this is true if and only if $\im(\pi^*)\cap \ker(\pi) = 0$. Lemma \ref{lemma:image-pi-dual} gives $\im(\pi^*) = Z^k$, and clearly $\ker(\pi) = B_k$. In other words, $\pi^\dagger$ exists exactly when $B_k \cap Z^k = 0$. Corollary \ref{cor:uniqueness-condition} finishes the proof.
\end{proof}


\subsection{Proofs of Theorem A and Theorem B}
\begin{proof}[Proof of Theorem~\ref{thm:TFAE-list}]
  First note that, since $Z_k \cap Z^k$ is the space of harmonic forms,
  (\ref*{it:A-HomologyIsomorphism}) is just an algebraic restatement of
  (\ref*{it:A-ExistenceUniqueness}). Omitting (\ref*{it:A-HomologyIsomorphism})
  from the list, we prove the equivalence of the remaining items in three
  parts.

	\textbf{Part 1: $(\ref*{it:A-ExistenceUniqueness})
	\Rightarrow (\ref*{it:A-Uniqueness})
	\Rightarrow (\ref*{it:A-UniquenessAt0})
	\Rightarrow (\ref*{it:A-TrivialIntersection})
	\Rightarrow (\ref*{it:A-PseudoinverseExistence}) 
	\Rightarrow (\ref*{it:A-ExistenceUniqueness}) $.}

		The implications $(\ref*{it:A-ExistenceUniqueness})
		\Rightarrow (\ref*{it:A-Uniqueness})
		\Rightarrow (\ref*{it:A-UniquenessAt0})$ are immediate.
		To see that (\ref*{it:A-UniquenessAt0}) $\Rightarrow$ (\ref*{it:A-TrivialIntersection}), recall from Proposition~\ref{prop:uniqueness-measurement}
    that the harmonic representatives of any given homology class form a torsor
    over $B_k\cap Z^k$, the set of harmonic representatives of the trivial
    homology class. If 0 is the only such representative, this torsor must be
    over the trivial space, forcing $B_k\cap Z^k = 0$.

		For (\ref*{it:A-TrivialIntersection}) $\Rightarrow$ (\ref*{it:A-PseudoinverseExistence}),
    $\pi^\dagger$ exists when $\pi\pi^*$ is invertible by
    Proposition~\ref{prop:pseudoinverse-of-surjective-map-existence-and-form},
    and since $\pi^*$ is injective, this holds exactly when $\im(\pi^*)\cap
    \ker(\pi) = 0$. Clearly $\ker(\pi) = B_k$, and
    Lemma~\ref{lemma:image-pi-dual} implies $\im(\pi^*) = Z^k$. Therefore
    $\im(\pi^*) \cap \ker(\pi) = B_k \cap Z^k = 0$ and so $\pi^\dagger$ exists.

    Finally, (\ref*{it:A-PseudoinverseExistence}) $\Rightarrow$
    (\ref*{it:A-ExistenceUniqueness}) since, by
    Theorem~\ref{thm:pseudoinverse-usefulness} and
    Lemma~\ref{lemma:pseudoinverse-existence-condition}, $\pi^{\dagger}(h)$ is
    the unique harmonic representative of $h \in H_k(X;\bF)$. 
	
	\textbf{Part 2:} (\ref*{it:A-TrivialIntersection}) and (\ref*{it:A-OrthogonalSum}) are equivalent.

		This follows from a dimension counting argument:
    \[
      \dim(Z^k) = \nullity(\partial^*_{k+1}) = \dim(C_k) -
    \rank(\partial_{k+1}^*) = \dim(C_k) - \rank(\partial_{k+1}) \, ,
  \]
  where $\nullity(\partial^*_{k+1}) = \dim(\ker(\partial_{k+1}^*))$. Thus
		$\dim(B_k)+\dim(Z^k) = \dim(C_k)$.
	
	\textbf{Part 3:} (\ref*{it:A-TrivialIntersection}), (\ref*{it:A-KernelOfComposition}), and (\ref*{it:A-ImageOfComposition}) are equivalent.
	
  Equivalence of (\ref*{it:A-TrivialIntersection}) and (\ref*{it:A-KernelOfComposition}) follows from the general fact that, for linear maps $f$ and $g$, $\ker(f\circ g) = \ker(g)$ if and only if $\ker(f) \cap \im(g) = 0$.
		Choosing $f = \partial_{k+1}^*$ and $g = \partial_{k+1}$ gives the result.

		Lastly, items 
		(\ref*{it:A-KernelOfComposition}) and
		(\ref*{it:A-ImageOfComposition}) are related by taking transposes. Indeed,
		\begin{align*}
		\ker(\partial_{k+1}^\ast \partial_{k+1}) = \ker(\partial_{k+1})
		& \iff
		\ker((\partial_{k+1}^\ast \partial_{k+1})^{**}) = \ker(\partial_{k+1}^{**})
		\\
		& \iff
		\im((\partial_{k+1}^\ast \partial_{k+1})^*)^\perp = \im(\partial_{k+1}^*)^\perp\\
		& \iff
		\im(\partial_{k+1}^*\partial_{k+1}) = \im(\partial_{k+1}^*).
		\end{align*}
		This completes the proof of equivalence.
	
	The final statement about $\pi^\dagger$ follows from Theorem~\ref{thm:pseudoinverse-usefulness}, using statement (\ref*{it:A-PseudoinverseExistence}) for the existence of $\pi^\dagger$ and statement (\ref*{it:A-Uniqueness}) for uniqueness of the representative.
	\end{proof}

  We finish the section with a proof of the Hodge decomposition.

  \begin{proof}[Proof of Theorem~\ref{thm:hodge-decomposition}]  
		Suppose $C_\bullet$ is both homologically and cohomologically harmonic in degree $k$. Applying Theorem~\ref{thm:TFAE-list} item (\ref*{it:A-OrthogonalSum}), we decompose $C_k$ as an inner direct sum $C_k = Z^k \oplus B_k$. These subspaces are orthogonal with respect to the standard symmetric bilinear form, since $B_k^\perp = Z^k$. Then applying the same part of Theorem~\ref{thm:TFAE-list} to $C^\bullet$ at degree $k$, we get $C_k = C^k = Z_k \oplus B^k$, which is again an orthogonal decomposition. Intersecting with $Z^k$ gives
		\begin{align*}
			Z^k = (Z^k \cap Z_k) \oplus (Z^k \cap B^k) = (Z^k \cap Z_k) \oplus B^k.
		\end{align*}
		Substituting into the first decomposition above, we obtain the orthogonal decomposition
		\begin{align*}
			C_k = Z^k \oplus B_k = (Z^k \cap Z_k) \oplus B^k\oplus B_k
		\end{align*}
		as desired.
		
		Conversely, suppose the above decomposition holds.
    Then in particular, $(Z_k \cap Z^k) \cap B_k = B_k \cap Z^k = 0$.
		Theorem~\ref{thm:TFAE-list} item (\ref*{it:A-TrivialIntersection}) implies $C_\bullet$ is homologically harmonic.
		A similar argument shows that $B^k \cap Z_k = \emptyset$, implying that $C_\bullet$ is cohomologically harmonic as well, completing the left hand implication.

		Next we show that $Z_k \cap Z^k = \ker \mathcal L_k$.
    Since $\mathcal L_k = \partial_{k+1}\partial_{k+1}^* + \partial_k^*\partial_k$, clearly $Z_k \cap Z^k \subseteq \ker \mathcal L_k$.
    For the opposite inclusion, suppose $\mathcal L_k x = 0$. Then $\partial_k^*\partial_k x = \partial_{k+1}\partial_{k+1}^*(-x) \in B_k\cap B^k = 0$. Applying item~\ref*{it:A-KernelOfComposition} to both $C_\bullet$ and $C^\bullet$ shows $x\in Z_k \cap Z^k$ as desired.

		We have already shown that such a decomposition is orthogonal, so it only remains to show that the natural maps from $(Z^k \cap Z_k)$ to $H_k$ and $H^k$ are both isomorphisms.
		This follows from applying Theorem~\ref{thm:TFAE-list} item (\ref*{it:A-HomologyIsomorphism}) to both $C_\bullet$ and $C^\bullet$.
  \end{proof}

\section{Harmonic Representatives on CW Complexes over $\F{p}$}\label{sec:rational-methods}
We have explored when a chain complex $C_\bullet$ is homologically harmonic in degree $k$. However, in the case that $C_\bullet = C_\bullet(X;\F{})$, it is not clear how changing the field $\F{}$ affects this property. In this section, we examine the special case where $\F{} = \F{p}$ is a finite field of prime order, examining for which $p$ a fixed CW complex is homologically harmonic over $\F{p}$. Our conditions are based on the integer $\Upsilon$, which is itself defined in terms of the following combinatorial structure.

\begin{definition}[{\cite[Definition 1.10]{catanzaro2017}}]\label{def:spanning-cotree}
	Let $X$ be a finite CW complex with dimension at least $k$. A \emph{$k$-dimensional spanning cotree} of $X$ is a subcomplex $L$ such that
    \begin{itemize}
        \item $X^{(k-1)}\subseteq L \subseteq X^{(k)}$
	    \item The inclusion $i_L : L \to X$ induces rational isomorphisms 
		\[ i_{L*} : H_{k}(L;\mathbb{Q}) \to H_{k}(X;\mathbb{Q}) \mbox{ and } 
		i_{L*}: H_{k-1}(L;\mathbb{Q}) \to H_{k-1}(X;\mathbb{Q}) \, .
	\]
    \end{itemize}
\end{definition}
\begin{remark}\label{rem:row-basis}
    Recall that the rows of the matrix $\partial_{k}:C_k(X) \to C_{k-1}(X)$ correspond to the $k$-cells of $X$. We could equivalently define a $k$-dimensional spanning cotree to be a subcomplex $X^{(k-1)}\subseteq L \subseteq X^{(k)}$ such that the $k$-cells \emph{not} in $L$ form a row basis for the matrix $\partial_{k}$.
\end{remark}
Spanning co-trees of the appropriate dimension always exist~\cite[Lemma 2.1]{catanzaro2017}.
Furthermore, the rational isomorphisms in their definition imply the relative integral homology group $H_k(X,L)$ is always finite for
any $k$-dimensional spanning cotree $L$.

\begin{definition}
    Given a $k$-dimensional spanning cotree $L$, the \emph{weight} of $L$ is
    \[a_L = |H_k(X,L)|.\]
\end{definition}

\begin{definition}\label{def:constant-D}
	Given a finite CW complex $X$ and $k\geq0$, define
    \[\Upsilon = \Upsilon(X,k) = \left( \sum_L a_L^2 \right)\prod_L a_L \, , \]
    where the sum and product are each over all $k$-dimensional spanning cotrees of $X$.
\end{definition}

\begin{definition}
    Let $\Z_{(p)} \subseteq \mathbb{Q}$ be the subring of rational numbers with denominators not divisible by $p$. Define $r_p:\Z_{(p)} \to \F p$ to be the ring map $r_p(a/b) = ab^{-1}\mod p$. Given a matrix $A$ with entries in $\Z_{(p)}$, define $R_pA$ to be the matrix obtained by applying $r_p$ entrywise to $A$.
\end{definition}



\begin{lemma}
    Let $p$ be a prime not dividing any denominators of the entries of $\pi^\dagger\pi$, and such that $H_k(X)$ has no $p$-torsion. Then the matrix $R_p(\pi^\dagger\pi)$ is the orthogonal projection map of $C_k(X;\mathbb{F}_p)$ onto $Z^k(X;\mathbb{F}_p)$ with respect to the standard bilinear form $b_k$ on $C_k(X;\mathbb{F}_p)$. That is,
    \begin{itemize}
        \item $R_p(\pi^\dagger\pi)$ is self-adjoint under $b_k$.
        \item $R_p(\pi^\dagger\pi)$ fixes its image.
        \item $\im(R_p(\pi^\dagger\pi)) = Z^k$.
    \end{itemize}
\end{lemma}

\begin{proof}

    First, it is clear that since $\pi^\dagger\pi$ is a symmetric matrix, so is $R_p(\pi^\dagger \pi)$. A symmetric matrix is self-adjoint under the standard bilinear form $b_k$ if and only if it is symmetric, so $R_p(\pi^\dagger \pi)$ is indeed self-adjoint.

    We show that $\im(R_p(\pi^\dagger\pi)) \subseteq Z^k(X;\mathbb{F}_p)$. If we let $c\in C_k(X;\mathbb{F}_p)$ and pick $c'\in C_k(X;\mathbb{Q})$ with $r_{p*}(c')=c$, then we have
    \begin{align*}
        R_p(\pi^\dagger\pi)(c) = r_{p*}(\pi^\dagger\pi(c')).
    \end{align*}
    Since $\pi^\dagger\pi(c')$ is always a cocycle, and $r_{p*}$ is a chain map, we get that $r_{p*}(\pi^\dagger\pi(c')) \in Z^k(X;\mathbb{F}_p)$.

    Finally, we show $R_p(\pi^\dagger\pi)$ fixes every $z\in Z^k(X;\mathbb{F}_p)$. We claim that there is a $z'\in Z^k(X;\mathbb{Z}[\Upsilon^{-1}])$ such that $r_{p*}(z') = z$. To see this, consider the short exact sequence of groups
    \[0 \to \Z[\Upsilon^{-1}] \xrightarrow{\cdot p} \Z[\Upsilon^{-1}] \xrightarrow{r_p} \mathbb{F}_p \to 0 \, ,\]
    and the induced short exact sequence of cochain complexes
    \[0 \to C^k(X;\Z[\Upsilon^{-1}]) \xrightarrow{\cdot p} C^k(X;\Z[\Upsilon^{-1}]) \xrightarrow{r_{p*}} C^k(X;\mathbb{F}_p) \to 0.\]
    The snake lemma gives us the exact sequence
    \[Z^k(X;\Z[\Upsilon^{-1}]) \xrightarrow{r_{p*}} Z^k(X;\mathbb{F}_p) \xrightarrow{\delta} C^{k+1}(X;\Z[\Upsilon^{-1}])/B^{k+1}(X;\Z[\Upsilon^{-1}]).\]
    To prove that our desired $z'$ exists, by exactness it suffices to show
    that $\delta = 0$. Indeed, $Z^k(X;\mathbb{F}_p)$ consists entirely of
    $p$-torsion, so it further suffices to show that
    $C^{k+1}(X;\Z[\Upsilon^{-1}])/B^{k+1}(X;\Z[\Upsilon^{-1}])$ has no
    $p$-torsion. A representative of a $p$-torsion element would be a chain
    $x\in C^{k+1}(X;\Z[\Upsilon^{-1}])$ with $px = \partial^* y$ for some $y$.
    Then $px$ is a cocycle, and hence so is $x$. We now have that $x$
    represents a $p$-torsion class of $H^{k+1}(X;\Z[\Upsilon^{-1}])$. By an application of the universal coefficient theorem, since $p\nmid \Upsilon$, the
    $p$-torsion of $H^{k+1}(X;\Z[\Upsilon^{-1}])$ is isomorphic to the $p$-torsion of
    $H_k(X)$, and we have assumed that the only such homology class is trivial.
    Therefore $\delta = 0$, and we can pick a cocycle $z'$ with $r_{p*}(z') =
    z$.

    \begin{align*}
        R_p(\pi^\dagger\pi)(z) = r_{p*}(\pi^\dagger\pi(z')) = r_{p*}(z') = z,
    \end{align*}
    which completes the proof.
\end{proof}

\begin{corollary}\label{cor:X-harmonic-projection-denominators}
    If $p$ does not divide the denominators of the entries of the matrix $\pi^\dagger \pi$ and $H_k(X)$ has no $p$-torsion, then $C_k(X;\mathbb{F}_p) = Z^k(X;\mathbb{F}_p) + B_k(X;\mathbb{F}_p)$, and $X$ is homologically harmonic in degree $k$ over $\mathbb{F}_p$.
\end{corollary}

\begin{proof}
    Pick any chain $c\in C_k(X;\mathbb{F}_p)$. Then
    \begin{align*}
        c = (c- R_p(\pi^\dagger\pi)(c)) + R_p(\pi^\dagger\pi)(c).
    \end{align*}
    By the previous result, $R_p(\pi^\dagger\pi)(c)$ is a cocycle. To show that $c- R_p(\pi^\dagger\pi)(c)$ is a boundary, we write
    \begin{align*}
        R_p(\pi^\dagger\pi)(c- R_p(\pi^\dagger\pi)(c))
        = R_p(\pi^\dagger\pi)(c) - R_p(\pi^\dagger\pi)\circ R_p(\pi^\dagger\pi)(c).
    \end{align*}
    Again by the previous result, $R_p(\pi^\dagger\pi)$ fixes its image, so the above expression is 0. Then $c- R_p(\pi^\dagger\pi)(c)$ is a boundary, since it is orthogonal to every cocycle $z$:
    \begin{align*}
        b_k(z,c- R_p(\pi^\dagger\pi)(c)) 
        & =
        b_k(z,c)- b_k(z, R_p(\pi^\dagger\pi)(c))
        \\ &
        = b_k(z,c)- b_k(R_p(\pi^\dagger\pi)(z), c)
        \\ &
        = b_k(z,c)- b_k(z, c)
        = 0.
    \end{align*}
    Hence $C_k(X;\mathbb{F}_p) = Z^k(X;\mathbb{F}_p) + B_k(X;\mathbb{F}_p)$. By Theorem~\ref{thm:TFAE-list}, it follows that $X$ is homologically harmonic in degree $k$ over $\mathbb{F}_p$.
\end{proof}

The conditions of Corollary~\ref{cor:X-harmonic-projection-denominators} depend on the denominators of $\pi^\dagger\pi$. The following proposition gives a combinatorial description of the denominators that can appear.

\begin{theorem}\label{thm:projection-entries}
    The entries of the matrix $\pi^\dagger\pi$ lie in $\Z[\Upsilon^{-1}]$.
\end{theorem}

\begin{proof}
  By clearing denominators
  (or choosing an appropriate basis for $C_k(X;\mathbb{Q})/B_k(X;\mathbb{Q})$),
  we may assume the entries of $\pi$ are integers. 
  Following~\cite[Thm A]{catanzaro2017} as adapted from the general construction
  of~\cite[Thm 2.1]{ben-tal_geometric_1990},
  we can give an explicit construction of $\pi^{\dagger}$. Let $m$ be the rank
  of $C_k(X;\mathbb{Q})/B_k(X;\mathbb{Q})$ and for a subset $S \subset \{1, 2,
  \ldots, |X_k|\}$ of cardinality $m$, let $\pi_S$ denote the
  $m \times m$ submatrix of $\pi$ whose rows correspond to indices in $S$. We restrict
  our focus to only those $S$ so that $\pi_S$ is invertible. Denote the inclusion
  corresponding to those rows by
  $i_S: \mathbb{Z}^m \to C_k(X)$, set $t_S = \det(\pi_S)^2$, 
  and finally set $\nabla = \sum_S t_S$. Then 
  \[
    \pi^{\dagger} = \tfrac{1}{\nabla} \sum_S t_S i_S(\pi_S)^{-1} \, .
  \]
  Under the correspondence described in~\cite[Thm A, Rem 2.7]{catanzaro2017}, each
  such $S$ corresponds to a spanning cotree $L$ and furthermore, $t_S = a_L^2$. 
  The inverse $i_S(\pi_S)^{-1}$ requires the invertibility of $a_L$ and thus
  by definition of $\Upsilon$, $\pi^{\dagger}$ has coefficients in $\Z[\Upsilon^{-1}]$.
  Hence the product $\pi^{\dagger}\pi$ has coefficients in $\Z[\Upsilon^{-1}]$.
\end{proof}

\section{Harmonic Representatives on Orientable Surfaces}\label{sec:orientable-surfaces}
A given CW complex $X$ may have many spanning cotrees of a given dimension, since the corresponding boundary map may have many row bases.
This can make $\Upsilon$ difficult to compute directly.
In the case that $X=M$ is a CW structure on an orientable surface, the formula for $\Upsilon$ can be simplified, making its computation simpler.
As we will show, for any spanning cotree $L$ of $M$, the relative homology group $H_1(M,L)$ is trivial,
in which case the formula for $\Upsilon$ simplifies to just the number of spanning cotrees. We take the convention that a surface is connected.

\subsection{The Dual Cell Structure}\mbox{}
Given a CW decomposition $M$ of an orientable surface, there are multiple
ways of constructing the dual complex $M^*$~\cite{hatcher_algebraic_2002}. 
A careful construction of the dual cell structure is beyond the scope of this paper. The salient points we need are the following.
\begin{itemize}
    \item There is a canonical bijection between the $k$-cells of $M$ and the $(2-k)$-cells of $M^*$ for $0\leq k \leq 2$. This lets us identify $C_k(M)$ and $C_{2-k}(M^*)$.
    \item Under the above identifications, $\partial_1^{M^*}:C_1(M^*) \to C_0(M^*)$ is the transpose of $\partial_2^M:C_2(M)\to C_1(M)$, and similarly $\partial_2^{M^*}:C_2(M^*) \to C_1(M^*)$ is the transpose of $\partial_1^M:C_1(M)\to C_0(M)$.
\end{itemize}

\subsection{Connections between Spanning Trees and Cotrees}

In Remark~\ref{rem:row-basis} we note that a spanning cotree can be defined as corresponding to the complement of a row basis for the appropriate boundary matrix. We can similarly define a $k$-dimensional spanning tree $T$ to be a subcomplex $X^{(k-1)}\subseteq T \subseteq X^{(k)}$ such that the $k$-cells of $T$ correspond to a column basis for $\partial_k$ (see~\cite{catanzaro_kirchhoffs_2015} for a homological definition). When $X$ is a connected graph and $k=1$, this definition aligns with the usual graph-theoretic notion of a spanning tree. When it is understood that $T$ is a $k$-dimensional spanning tree, we often identify $T$ with its set of $k$-cells $T_k\subset X_k$.

Given a subcomplex $X^{(k-1)}\subseteq L\subseteq X^{(k)}$ such as a spanning tree or cotree, we use $L^\perp$ to denote the subcomplex $X^{(k-1)}\subseteq L^\perp\subseteq X^{(k)}$ whose $k$-cells are $X_k \setminus L_k$. If we identify $L$ with its set of $k$-cells, then $L^\perp$ is simply the complement of $L$. In the special case that $L$ is a 1-dimensional spanning cotree of a CW structure $M$ on an orientable manifold, $L^\perp$ can be thought of as a 1-dimensional spanning tree of the dual cell structure $M^\star$. This is because if $L$ is a spanning cotree, then its 1-cells correspond to the complement of a row basis for $\partial^M_2$. Hence the 1-cells of $L^\perp$ correspond to a column basis of $(\partial^M_2)^*=\partial^{M^*}_1$.

\begin{definition}
    Given a $k$-dimensional spanning tree $T$, its \emph{weight} is defined to be
    \[\theta_T = |H_{k-1}(T)_t|.\]
\end{definition}

The following result can simplify calculations of weights of spanning trees and
cotrees. Let $X_{k,k-2}:=X^{(k)}/X^{(k-2)}$, the $k$-skeleton of $X$ with the
$(k-2)$-skeleton collapsed to a point.
\begin{proposition}[{\cite[Lemma 4.2]{catanzaro2017}}]\label{prop:simplifying-weights}
    For any spanning tree $T$ or spanning cotree $L$,
    \[\theta_{T} = \theta_{T_{k,k-2}} \qquad a_L = a_{L_{k,k-2}}.\]
\end{proposition}

The following formula from the same paper will help us relate these weights to the restricted Laplacian.
\begin{proposition}[{\cite[Theorem C]{catanzaro2017}} Higher Matrix-Tree Theorem]
  \label{prop:higher_matrix_tree}
    Let $\widehat{\mathcal L}_k:B_k(X) \to B_k(X)$ be the Laplacian restricted to the $k$-boundaries. Then
    \[\det \widehat{\mathcal L}_k = \frac{1}{\theta_X^2}\left(\sum_L a_L^2\right)\left(\sum_T \theta_T^2\right) \, ,\]
    where the first sum ranges over all $k$-dimensional spanning cotrees of $X$, and the second ranges over $(k+1)$-dimensional spanning trees. 
\end{proposition}

\subsection{Smith Normal Form}\label{subsec:smith-normal-form}
Given an $m\times n$ integer matrix $A$, its \emph{Smith normal form} is another $m\times n$ matrix $SAT$ where $S\in GL_m(\mathbb Z)$, $T \in GL_n(\mathbb Z)$, and $SAT$ is of the form
\[\begin{pmatrix}
    a_1 & 0 & 0 & \dots&  0 \\
    0 & a_2 & 0 & \dots&  0 \\
    0 & 0 & a_3 & \dots&  0 \\
    \vdots & \vdots & \vdots & \ddots \\
    0 & 0 & 0 && 0
\end{pmatrix}\]
with $a_i \, | \,a_{i+1}$ for all $i$. The Smith normal form of $A$ always exists and is unique up to the signs of the nonzero entries. The crucial properties we will be using are that the Smith normal forms of $A$ and $A^T$ are themselves transposes, and that $|\coker(A)| = |\coker(SAT)| = |a_1\dots a_k|$.

\begin{lemma}\label{lemma:tree-weights-smith}
    Let $T$ be a $k$-dimensional spanning tree of $X$, and let $t_1,\dots, t_n$ be the nonzero entries of the Smith normal form of the composition
    $C_k(T) \hookrightarrow C_k(X) \xto{\partial} C_{k-1}(X).$
    Then \[\theta_T = t_1\cdots t_n.\]
    Similarly, if $L$ is a $(k-1)$-dimensional spanning cotree of $X$, and $\ell_1,\dots, \ell_n$ are the nonzero entries of the Smith normal form of $C_k(X) \xto{\partial} C_{k-1}(X) \twoheadrightarrow C_{k-1}(X)/C_{k-1}(L)$, then
    \[a_L = \ell_1\cdots \ell_n.\]
\end{lemma}
\begin{remark}
    The matrix for $C_k(T) \hookrightarrow C_k(X) \xto{\partial} C_{k-1}(X)$ is the boundary matrix for $X$ restricted to the columns corresponding to the $k$-cells of $T$, precisely the submatrix mentioned in our definition of a spanning tree. Similarly, the matrix for $C_k(X) \xto{\partial} C_{k-1}(X) \twoheadrightarrow C_{k-1}(X)/C_{k-1}(L)$ is the boundary matrix for $X$ after removing the rows corresponding to the $(k-1)$-cells of $L$, which is the submatrix mentioned by Remark~\ref{rem:row-basis}.
\end{remark}
\begin{proof}[Proof of 6.4]
    We handle the spanning tree first. We know that $\theta_T = \theta_{T_{k,k-2}} = |H_{k-1}(T)_t|$ by Proposition~\ref{prop:simplifying-weights}, which is the size of the torsion part of the $(k-1)$-homology of the chain complex
    \[\dots \to 0 \to C_k(T) \xto{\partial^T_k} C_{k-1}(T) \to 0 \to \dots \, , \]
    whose homology is $\coker(\partial^T_k)$. Note that $C_{k-1}(T) = C_{k-1}(X)$, and the map $\partial^T_k$ is the boundary map $\partial:C_k(X) \to C_{k-1}(X)$ restricted to $C_k(T)$.
    Therefore we can factor the map as the composition in the statement of the lemma.
    Our remark about the size of the cokernel in Subsection~\ref{subsec:smith-normal-form} finishes this part.

    The spanning cotree is handled similarly. Using the fact that $a_{L} = a_{L_{k,k-2}}$, it suffices to calculate the $k-1$ homology of the relative chain complex
    \[\dots \to 0 \to C_k(X)/C_k(L) \xto{\overline \partial_k} C_{k-1}(X)/C_{k-1}(L) \to 0 \to \dots,\]
    which is $\coker(\overline\partial_k)$. We already know this cokernel is finite, since its order is $a_L$, so its free part must be trivial.
    Note that $C_k(L) = 0$, so the domain of this map is $C_k(X)$.
    It then factors as the composition in the statement of the lemma.
    Since the entire cokernel is torsion, we can again use our remark about the size of the cokernel in Subsection~\ref{subsec:smith-normal-form} to finish the proof.
\end{proof}

For the statement of the following lemma, recall that if $L$ is a 1-dimensional spanning cotree of $M$, we can regard $L^\perp$ as a 1-dimensional spanning tree of $M^*$.
\begin{lemma}\label{lemma:tree-cotree-weights}
    If $L$ is a spanning cotree for $M$ of any dimension $0\leq k \leq 2$, then $a_L = \theta_{L^\perp}$.
\end{lemma}
\begin{proof}
    The matrices for $C_{k+1}(M) \xto{\partial^M_{k+1}} C_k(M)
    \twoheadrightarrow C_k(M)/C_k(L)$ and $C_k(L^\perp) \hookrightarrow
    C_k(M^*) \xto{\partial^{M^*}_k} _{k-1}(X)$ are transposes of each
    other.  Indeed, the former is the matrix for $\partial_{k+1}^M$ with the
    rows corresponding to the 1-cells of $L$ removed, while the latter is the
    matrix for $\partial^{M^*}_k$ (the transpose of $\partial_{k+1}^M$) with
    the columns corresponding to the 1-cells of $L$ removed.  Since these
    matrices are transposes, we conclude that their Smith normal forms are
    transposes as well, and in particular have the same nonzero entries.
    Lemma~\ref{lemma:tree-weights-smith} finishes the proof.
\end{proof}

\begin{lemma}\label{lemma:cotrees-have-weight-one}
    For any 1-dimensional spanning cotree $L$ of $M$, $a_L = 1$.
\end{lemma}
\begin{proof}
    By Lemma~\ref{lemma:tree-cotree-weights}, it suffices to prove that every 1-dimensional spanning tree $T$ of $M^*$ has weight 1.
    This follows since zero-degree homology is always free.
\end{proof}

\begin{theorem}\label{thm:orientable-surface-D-formula}
    Let $M$ be a CW structure on an orientable surface. Then $\Upsilon(M,1)$ is the number of 1-dimensional spanning cotrees of $M$, which equals
    \[\Upsilon(M,1) = \det \widehat{\mathcal L}_1/|M_2| \]
    where $\widehat{\mathcal L}_1:B_1(X) \to B_1(X)$ is the restricted Laplacian, and $|M_2|$ is the number of 2-cells of $M$.
\end{theorem}
\begin{proof}
    The fact that $\Upsilon(M,1)$ equals the number of spanning cotrees of $M$ follows immediately from Lemma~\ref{lemma:cotrees-have-weight-one} and the definition of $\Upsilon(M,1)$:
    \begin{align*}
        \Upsilon(M,1) = \left( \sum_L |H_1(X,L)|^2 \right)\prod_L |H_1(X,L)| = \left( \sum_L a_L^2 \right)\prod_L a_L = \sum_L 1.
    \end{align*}
    Rearranging the equation of Proposition~\ref{prop:higher_matrix_tree} gives the expression
    \[\Upsilon(M,1) = \frac{\theta^2_X\det \widehat{\mathcal{L}}_1}{\sum_T \theta_T^2}\]
    Since $X$ is an orientable surface, its second homology has no torsion and $\theta_X = 1$. Additionally, we can rewrite each $\theta_T$ as $a_{T^\perp}$, the weight of a zero-dimensional spanning cotree in $M^*$. We can derive from the definition of a spanning cotree that any spanning cotree of a connected complex consists of a single vertex. Hence $\theta_T = a_{L^\perp} = 1$. Plugging these into the equation above gives
    \[\Upsilon(M,1) = \frac{\det\widehat{\mathcal{L}}_1}{|(M^*)_0|} = \frac{\det\widehat{\mathcal{L}}_1}{|M_2|}. \qedhere\]

\end{proof}

\begin{corollary}\label{cor:orientable-surface-tidy-existence}
  Let $M$ be a CW structure on an orientable surface and $p$ a prime not dividing $\det\widehat{\mathcal L}_1/|M_2|$. Then $M$ is homologically harmonic in degree $1$ over $\F p$.
\end{corollary}

\section{Experiments}\label{sec:experiments}
In this section we present visualizations of harmonic representatives over various fields.
Much of the underlying SageMath code from this section is available as a Python package called simpleHarmony.
The source code, installation instructions, and examples can be found on github~\cite{simpleharm}.
Our examples come from two common applications of simplicial complexes: triangulations of surfaces and point cloud datasets.
There are several methods of constructing a simplicial complex on the latter, including the \u Cech complex and alpha complex.
We elect to use the Vietoris-Rips complex, which is the clique complex on a proximity graph of the data points.
See~\cite{ghrist2008} for a friendly introduction to \u Cech and Vietoris-Rips complexes, and their applications to persistent homology.

With a simplicial complex $X$ in hand, we manually choose an integral cycle $z \in
Z(X;\mathbb{Z})$ which appears to represent a topological feature of interest.
We rationalize this cycle by mapping it to $Z(X;\mathbb{Q})$, where we then compute and plot its
harmonic representative using the pseudoinverse from
Theorem~\ref{thm:pseudoinverse-usefulness}.  For comparison, we also 
map the integral cycle to $Z(X;\F{p})$ for some prime $p$ and calculate its
harmonic representative using the same method.  As highlighted by
Lemma~\ref{lemma:pseudoinverse-existence-condition}, the pseudoinverse 
need not exist over $\F{p}$ for all primes $p$, so we 
search for the smallest prime $p$ such that the necessary pseudoinverse exists.
After finding such a prime, we proceed to compute and plot the harmonic
representative.


\subsection{Minimal Triangulation of the Torus}

\begin{figure}
    \begin{tikzpicture}
        \draw[] (0,0) -- (0,1) -- (0,2) -- (0,3);
        \draw[] (0,3) -- (1,3) -- (2,3) -- (3,3);
        \draw[] (3,3) -- (3,2)  -- (3,1)  -- (3,0);
        \draw[] (3,0) -- (2,0)-- (1,0)-- (0,0);
        \draw[] (0,2) -- (1,3) -- (0,1) -- (2,2) -- (1,1) -- (3,2) -- (2,0) -- (3,1);
        \draw[] (1,1)--(2,2);
        \draw[] (0,1) -- (1,1) -- (0,0);
        \draw[] (1,0) -- (1,1) -- (2,0);
        \draw[] (3,3) -- (2,2) -- (3,2);
        \draw[] (2,3) -- (2,2) -- (1,3);
        \node at (0,0) [circle,draw=black,fill=black,inner sep=0.5mm] {};
        \node at (1,0) [circle,draw=black,fill=black,inner sep=0.5mm] {};
        \node at (2,0) [circle,draw=black,fill=black,inner sep=0.5mm] {};
        \node at (3,0) [circle,draw=black,fill=black,inner sep=0.5mm] {};
        \node at (0,1) [circle,draw=black,fill=black,inner sep=0.5mm] {};
        \node at (1,1) [circle,draw=black,fill=black,inner sep=0.5mm] {};
        \node at (3,1) [circle,draw=black,fill=black,inner sep=0.5mm] {};
        \node at (0,2) [circle,draw=black,fill=black,inner sep=0.5mm] {};
        \node at (2,2) [circle,draw=black,fill=black,inner sep=0.5mm] {};
        \node at (3,2) [circle,draw=black,fill=black,inner sep=0.5mm] {};
        \node at (0,3) [circle,draw=black,fill=black,inner sep=0.5mm] {};
        \node at (1,3) [circle,draw=black,fill=black,inner sep=0.5mm] {};
        \node at (2,3) [circle,draw=black,fill=black,inner sep=0.5mm] {};
        \node at (3,3) [circle,draw=black,fill=black,inner sep=0.5mm] {};
    \end{tikzpicture}

    \begin{tikzpicture}
        \draw[color = blue, very thick] (0,0) -- (0,1) -- (0,2) -- (0,3);
        \draw[color = blue, very thick] (3,3) -- (3,2)  -- (3,1)  -- (3,0);
        \draw[ultra thin, dashed] (0,3) -- (1,3) -- (2,3) -- (3,3);
        \draw[ultra thin, dashed] (3,0) -- (2,0)-- (1,0)-- (0,0);
        \draw[ultra thin, dashed] (0,2) -- (1,3) -- (0,1) -- (2,2) -- (1,1) -- (3,2) -- (2,0) -- (3,1);
        \draw[ultra thin, dashed] (1,1)--(2,2);
        \draw[ultra thin, dashed] (0,1) -- (1,1) -- (0,0);
        \draw[ultra thin, dashed] (1,0) -- (1,1) -- (2,0);
        \draw[ultra thin, dashed] (3,3) -- (2,2) -- (3,2);
        \draw[ultra thin, dashed] (2,3) -- (2,2) -- (1,3);
        \node at (0,0) [circle,draw=black,fill=black,inner sep=0.5mm] {};
        \node at (1,0) [circle,draw=black,fill=black,inner sep=0.5mm] {};
        \node at (2,0) [circle,draw=black,fill=black,inner sep=0.5mm] {};
        \node at (3,0) [circle,draw=black,fill=black,inner sep=0.5mm] {};
        \node at (0,1) [circle,draw=black,fill=black,inner sep=0.5mm] {};
        \node at (1,1) [circle,draw=black,fill=black,inner sep=0.5mm] {};
        \node at (3,1) [circle,draw=black,fill=black,inner sep=0.5mm] {};
        \node at (0,2) [circle,draw=black,fill=black,inner sep=0.5mm] {};
        \node at (2,2) [circle,draw=black,fill=black,inner sep=0.5mm] {};
        \node at (3,2) [circle,draw=black,fill=black,inner sep=0.5mm] {};
        \node at (0,3) [circle,draw=black,fill=black,inner sep=0.5mm] {};
        \node at (1,3) [circle,draw=black,fill=black,inner sep=0.5mm] {};
        \node at (2,3) [circle,draw=black,fill=black,inner sep=0.5mm] {};
        \node at (3,3) [circle,draw=black,fill=black,inner sep=0.5mm] {};
    \end{tikzpicture}
    \qquad
    \begin{tikzpicture}
        \draw[color = blue, thick] (0,3) -- (0,2) -- (1,3) -- (0,1) -- (2,2) -- (1,1) -- (3,2) -- (2,0) -- (3,1) -- (3,0);
        \draw[color = blue, thick] (1,3) -- (2,3) -- (2,2) -- (3,3) -- (3,2);
        \draw[color = blue, thick] (0,1) -- (0,0) -- (1,1) -- (1,0) -- (2,0);
        \draw[ultra thin, dashed] (0,3) -- (1,3) -- (2,2) -- (3,2) -- (3,1);
        \draw[ultra thin, dashed] (0,2) -- (0,1) -- (1,1) -- (2,0) -- (3,0);
        \draw[ultra thin, dashed] (2,3) -- (3,3);
        \draw[ultra thin, dashed] (0,0) -- (1,0);
        \node at (0,0) [circle,draw=black,fill=black,inner sep=0.5mm] {};
        \node at (1,0) [circle,draw=black,fill=black,inner sep=0.5mm] {};
        \node at (2,0) [circle,draw=black,fill=black,inner sep=0.5mm] {};
        \node at (3,0) [circle,draw=black,fill=black,inner sep=0.5mm] {};
        \node at (0,1) [circle,draw=black,fill=black,inner sep=0.5mm] {};
        \node at (1,1) [circle,draw=black,fill=black,inner sep=0.5mm] {};
        \node at (3,1) [circle,draw=black,fill=black,inner sep=0.5mm] {};
        \node at (0,2) [circle,draw=black,fill=black,inner sep=0.5mm] {};
        \node at (2,2) [circle,draw=black,fill=black,inner sep=0.5mm] {};
        \node at (3,2) [circle,draw=black,fill=black,inner sep=0.5mm] {};
        \node at (0,3) [circle,draw=black,fill=black,inner sep=0.5mm] {};
        \node at (1,3) [circle,draw=black,fill=black,inner sep=0.5mm] {};
        \node at (2,3) [circle,draw=black,fill=black,inner sep=0.5mm] {};
        \node at (3,3) [circle,draw=black,fill=black,inner sep=0.5mm] {};
    \end{tikzpicture}
    \caption{ \label{fig:minimalTriangulation}
      A harmonic representative on the torus.
    (Top) A minimal triangulation $X$ of the torus, with 7 vertices, 21 edges, and
    14 faces. The vertices and edges along the top edge of the square are
    identified with those on the bottom, and similarly with the left and right
    edges. (Lower left) The cycle $x \in Z_1(X;\F2)$. 
    The edges in its support are emphasized in blue, with all other
    edges drawn dashed. (Lower right) The unique harmonic representative of the
    homology class $[x]$, drawn in the same fashion.}
  \end{figure}
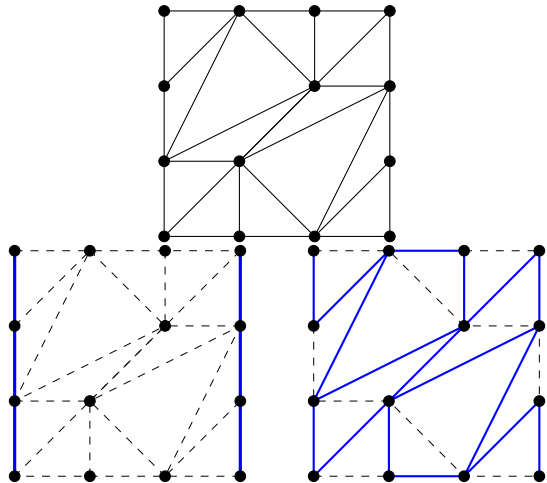

Consider the minimal triangulation of the torus depicted in
Figure~\ref{fig:minimalTriangulation} (top). We examine the first degree homology and
harmonic representatives of this simplicial complex $X$ over the field $\F2$.
Using code written in sage, we verify that $X$ is homologically harmonic over
$\F2$ in degree 1. This allows us to pick a particular cycle $x \in Z_1(X;\F2)$ (Figure~\ref{fig:minimalTriangulation}, bottom left)
and compute the unique harmonic representative of $[x]$ (Figure~\ref{fig:minimalTriangulation}, bottom right). We perform
this computation by calculating the pseudoinverse from
Theorem~\ref{thm:TFAE-list} statement (\ref{it:A-PseudoinverseExistence}) and applying it to $[x]$.

Alternatively, we can use the methods of Section~\ref{sec:rational-methods} to
find the same representative.
Indeed, we use sage to compute the determinant of the restricted Laplacian and, using Theorem~\ref{thm:fixed-X-variable-p}, arrive at $\Upsilon(X,1) =  3\cdot 7^5 = 50421$.
Since $\Upsilon$ is odd and $p=2$, Theorem~\ref{thm:fixed-X-variable-p} also ensures that $X$ is homologically harmonic in degree $k$ over $\F{p}$.
Furthermore, Theorem~\ref{thm:projection-entries} ensures that the entries of the rational matrix $\pi^\dagger \pi$ fall within $\mathbb Z[\Upsilon^{-1}]$. Applying $R_p(\pi^\dagger \pi)$ to $x$ then gives the same $\F2$-harmonic representative as before, showing that the two
methods do indeed agree as expected.

\subsection{Point Cloud Sampled from a Lemniscate}

\begin{figure}
    \includegraphics[width=6.4cm]{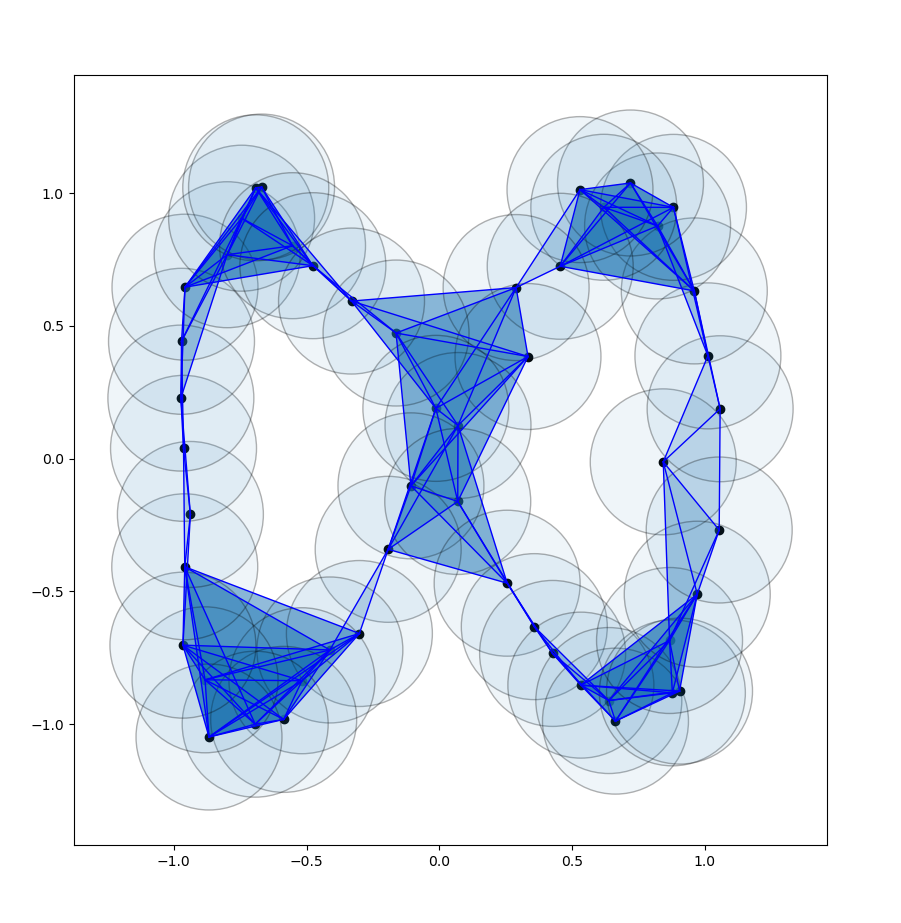}
    \includegraphics[width=6.4cm]{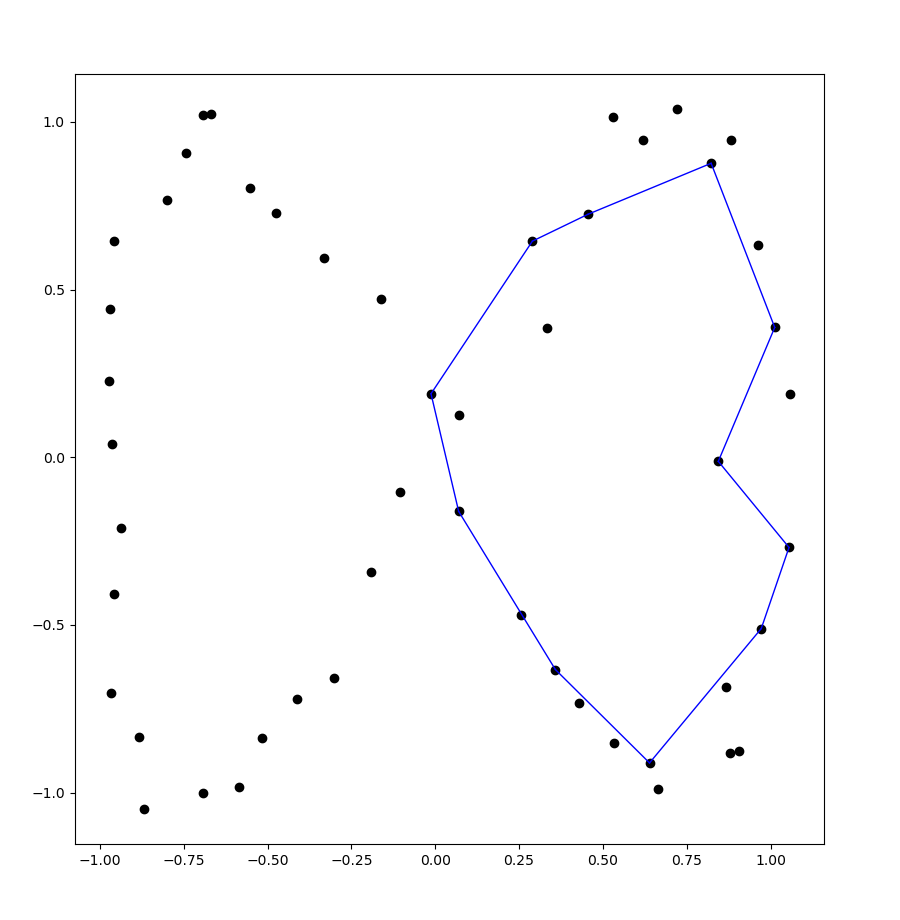}
    \includegraphics[width=6.4cm]{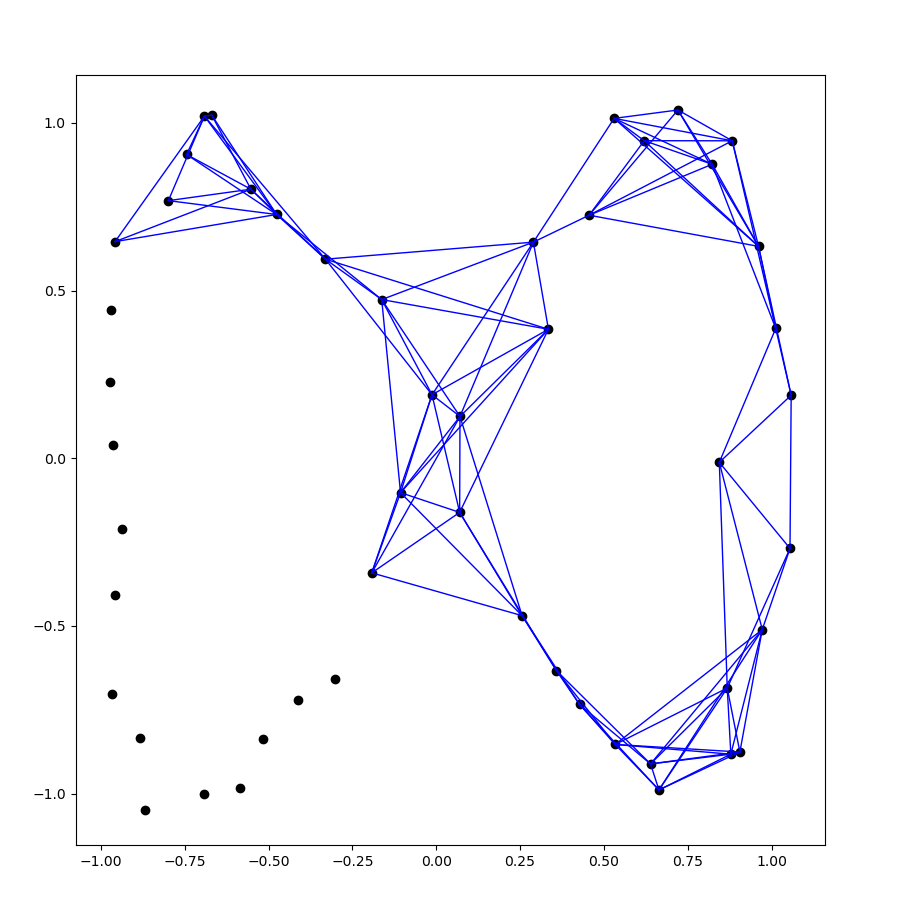}
    \includegraphics[width=6.4cm]{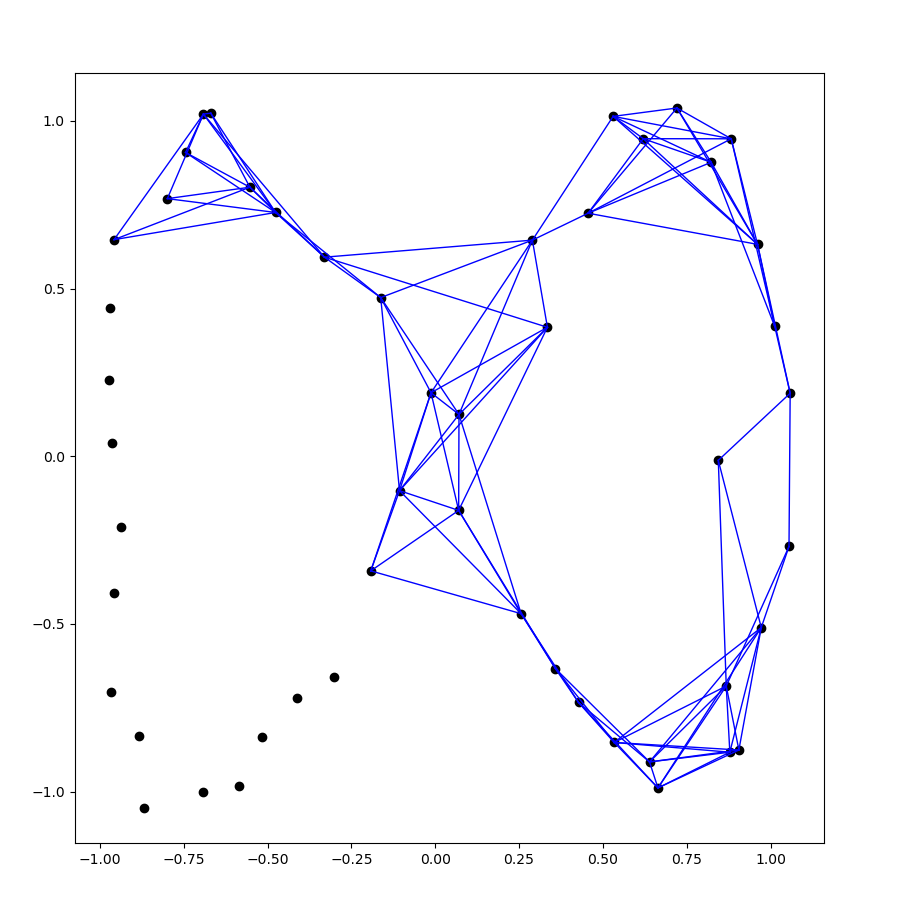}
    \includegraphics[width=6.4cm]{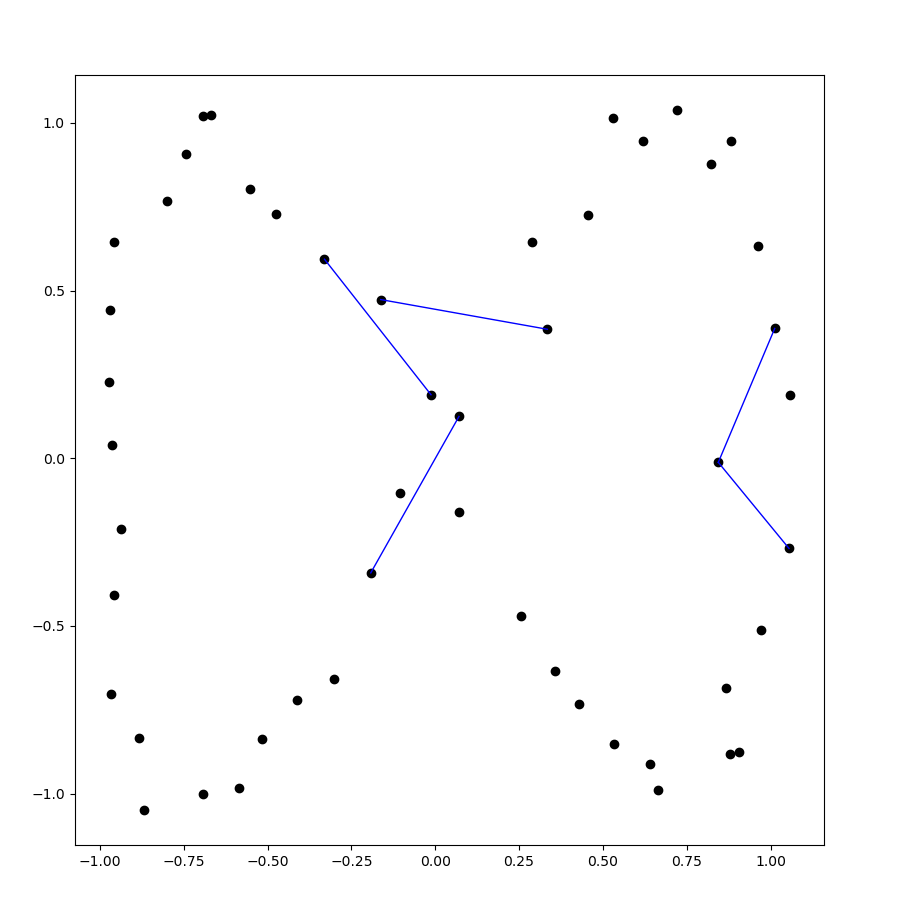}
    \caption{\label{fig:figure-eight} A Vietoris-Rips complex on a point cloud dataset sampled from a lemniscate with noise (upper left), a chosen integral cycle in the Vietoris-Rips complex (upper right), the corresponding harmonic forms over $\mathbb{Q}$ and $\F5$ (middle left and right), and the symmetric difference of those two supports.}
\end{figure}

For this experiment, we randomly sample 50 points from a lemniscate (a self-intersecting smooth curve homeomorphic to the wedge product $S^1\vee S^1$) embedded within $\mathbb{R}^2$ with Gaussian noise.
We construct the Vietoris-Rips complex $X$ on this
dataset with radius $0.45$ (Figure~\ref{fig:figure-eight}, top left).
We pick an integral cycle of interest $z\in H_1(X;\mathbb Z)$, by selecting
a cycle whose support wraps around the right hand
side (Figure~\ref{fig:figure-eight}, top right). We then consider $z$ as a cycle over both $\mathbb{Q}$ and
$\F5$, computing its harmonic representative in each case (Figure~\ref{fig:figure-eight}, middle left and right). In the latter case,
we achieve this with the pseudoinverse from
Theorem~\ref{thm:TFAE-list} item (\ref{it:A-PseudoinverseExistence}).
To highlight the similarity between these harmonic representatives, we also plot the symmetric difference of their supports (Figure~\ref{fig:figure-eight}, bottom).

In this case, the harmonic representatives over both $\mathbb{Q}$ and $\F5$
seem to visualize the initial geometric feature, with the support of each lying
most heavily on the right hand circle. They both localize the topological
phenomenon we are trying to isolate. This highlights one potential application
of harmonic representatives over finite fields to topological data analysis,
allowing practitioners another way to visualize persistent homology classes.

\subsection{Point Cloud Sampled from $S^1 \vee S^2 \vee S^1$}
\begin{figure}
    \includegraphics[width=7.5cm]{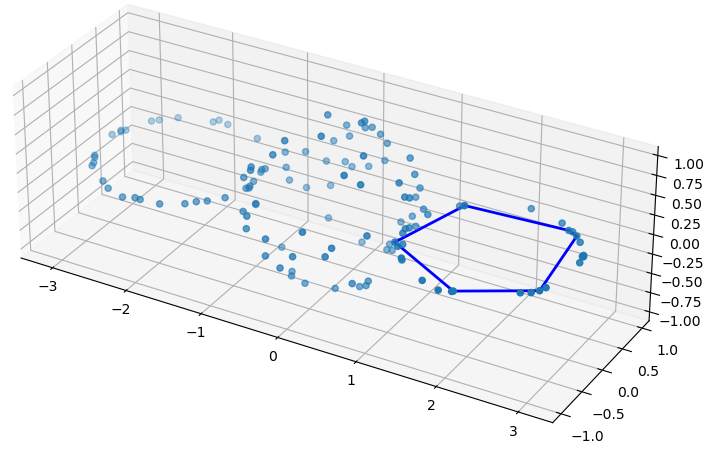}
    \includegraphics[width=7.5cm]{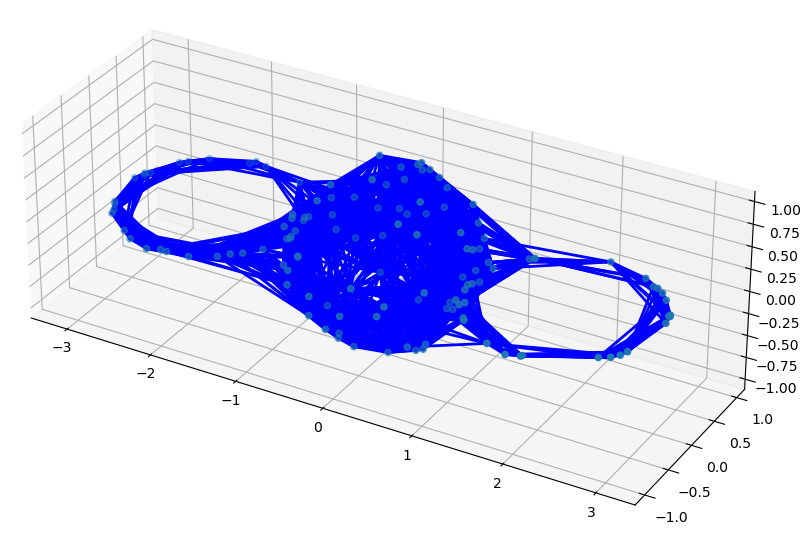}
    \includegraphics[width=7.5cm]{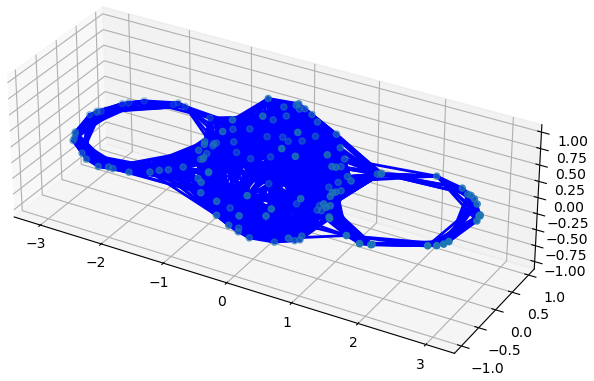}
    \includegraphics[width=7.5cm]{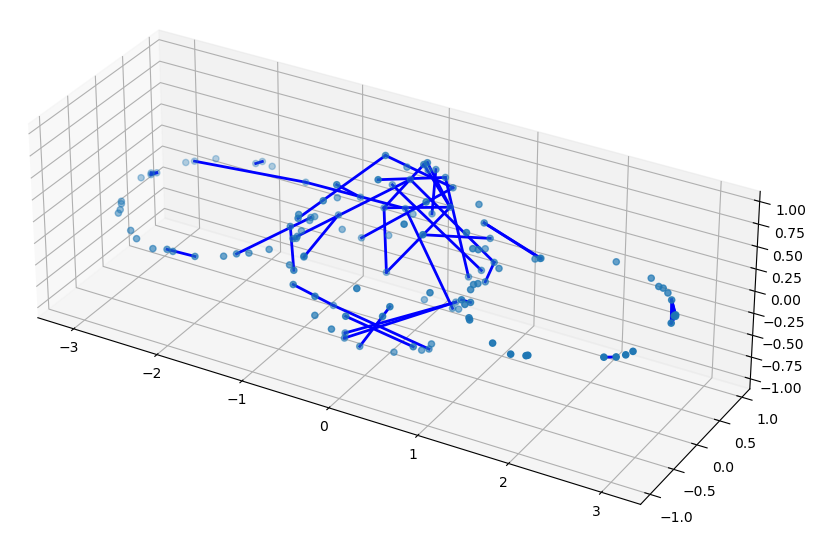}
    \caption{\label{fig:wedge-sphere} An integral cycle in the Vietoris-Rips complex on a point cloud sampled from $S^1 \vee S^2 \vee S^1$ with noise (upper left), the corresponding harmonic representatives over $\mathbb{Q}$ and $\F{47}$ (upper right and lower left), and the symmetric difference of those two supports.}
\end{figure}

We next run a similar experiment on a larger dataset in $\mathbb{R}^3$, inspired by \cite[Section 3.1]{Ebli2019}, wherein we sample 130 points from $S^1 \vee S^2 \vee S^1$, a sphere with two circles attached at antipodal points.
Of those 130 points, we take 70 from the sphere and 30 from each circle, with Gaussian noise of standard deviation $0.02$.
We construct the Vietoris-Rips complex $X$ on this dataset with radius $1.3$ and choose an integral cycle $z \in H_1(X;\mathbb{Z})$ whose support encompasses the circle attached to the right hand side of the sphere (Figure~\ref{fig:wedge-sphere}, top left).
As in the previous experiment, we plot the harmonic representatives of $z$ when considered as a cycle over $\mathbb{Q}$ and a finite field (Figure~\ref{fig:wedge-sphere}, top right, bottom left).
This time, the smallest field over which our calculation succeeded was $\F{47}$, showing that $X$ is not homologically harmonic over any smaller primes.

Notice that the support of the harmonic representative over the finite field is much larger in this case than the previous experiment.
Indeed, of the 1948 1-simplices in the Vietoris-Rips complex, 1893 have a nonzero coefficient in this cycle, which is not far from the 1907 that we would expect from a randomly chosen 1-chain.
Note, however, that the support of the rational harmonic form is similarly large, as shown by the relatively sparse symmetric difference of the supports (Figure~\ref{fig:wedge-sphere}, bottom right).
These results also differ from the previous experiment in that we require a much larger finite field before the statements of Theorem~A hold.

\bibliographystyle{plain}
\bibliography{HarmonicReps.bib}

\end{document}